\documentclass[11pt,twoside,english]{article}
\usepackage{mathptmx}
\usepackage[T1]{fontenc}
\usepackage[utf8]{inputenc}
\usepackage[a4paper]{geometry}
\usepackage{xcolor}
\usepackage{babel}
\usepackage{amsmath}
\usepackage{amsthm}
\usepackage{amssymb}
\usepackage{esint}
\usepackage[unicode=true,
 bookmarks=true,bookmarksnumbered=true,bookmarksopen=false,
 breaklinks=false,pdfborder={0 0 1},backref=false,colorlinks=true]
 {hyperref}
\hypersetup{
 pdfauthor={G Xi}}
\usepackage{breakurl}

\makeatletter
\theoremstyle{plain}
\newtheorem{thm}{\protect\theoremname}
  \theoremstyle{plain}
  \newtheorem{lem}[thm]{\protect\lemmaname}
  \theoremstyle{plain}
  \newtheorem{cor}[thm]{\protect\corollaryname}
  \theoremstyle{plain}
  \newtheorem{prop}[thm]{\protect\propositionname}
  \theoremstyle{remark}
  \newtheorem*{rem*}{\protect\remarkname}
  \theoremstyle{remark}
  \newtheorem{rem}[thm]{\protect\remarkname}

\usepackage{babel}
\allowdisplaybreaks[4]

\date{}

  \providecommand{\corollaryname}{Corollary}
  
  \providecommand{\lemmaname}{Lemma}
\providecommand{\theoremname}{Theorem}

\makeatother

  \providecommand{\corollaryname}{Corollary}
  \providecommand{\lemmaname}{Lemma}
  \providecommand{\propositionname}{Proposition}
  \providecommand{\remarkname}{Remark}
\providecommand{\theoremname}{Theorem}

\begin{document}
\global\long\def\divg{{\rm div}\,}

\global\long\def\curl{{\rm curl}\,}

\global\long\def\rt{\mathbb{R}^{3}}

\global\long\def\rn{\mathbb{R}^{n}}

\title{Parabolic equations with divergence-free drift in space $L_{t}^{l}L_{x}^{q}$}

\author{Zhongmin Qian\thanks{Research supported partly by an ERC grant Esig ID 291244. Mathematical
Institute, University of Oxford, OX2 6GG, England. Email: qianz@maths.ox.ac.uk} \ and\ Guangyu Xi\thanks{This work was supported by the Engineering and Physical Sciences Research
Council {[}EP/L015811/1{]}. Mathematical Institute, University of
Oxford, OX2 6GG, England. Email: guangyu.xi@maths.ox.ac.uk}}
\maketitle
\begin{abstract}
In this paper we study the fundamental solution $\varGamma(t,x;\tau,\xi)$
of the parabolic operator $L_{t}=\partial_{t}-\Delta+b(t,x)\cdot\nabla$,
where for every $t$, $b(t,\cdot)$ is a divergence-free vector field,
and we consider the case that $b$ belongs to the Lebesgue space $L^{l}\left(0,T;L^{q}\left(\mathbb{R}^{n}\right)\right)$.
The regularity of weak solutions to the parabolic equation $L_{t}u=0$
depends critically on the value of the parabolic exponent $\gamma=\frac{2}{l}+\frac{n}{q}$.
Without the divergence-free condition on $b$, the regularity of weak
solutions has been established when $\gamma\leq1$, and the heat kernel
estimate has been obtained as well, except for the case that $l=\infty,q=n$.
The regularity of weak solutions was deemed not true for the critical
case $L^{\infty}\left(0,T;L^{n}\left(\mathbb{R}^{n}\right)\right)$
for a general $b$, while it is true for the divergence-free case,
and a written proof can be deduced from the results in \cite{semenov2006regularity}.
One of the results obtained in the present paper establishes the Aronson
type estimate for critical and supercritical cases and for vector
fields $b$ which are divergence-free. We will prove the best possible
lower and upper bounds for the fundamental solution one can derive
under the current approach. The significance of the divergence-free
condition enters the study of parabolic equations rather recently,
mainly due to the discovery of the compensated compactness. The interest
for the study of such parabolic equations comes from its connections
with Leray's weak solutions of the Navier-Stokes equations and the
Taylor diffusion associated with a vector field where the heat operator
$L_{t}$ appears naturally.

\medskip

\emph{key words}: Aronson estimate, divergence-free vector field,
Harnack inequality, parabolic equation, weak solution 

\newpage{}
\end{abstract}

\section{Introduction}

We consider parabolic equations of second order with singular divergence-free
drift 
\begin{equation}
\partial_{t}u(t,x)-\sum_{i,j=1}^{n}\partial_{x_{i}}(a_{ij}(t,x)\partial_{x_{j}}u(t,x))+\sum_{i=1}^{n}b_{i}(t,x)\partial_{x_{i}}u(t,x)=0,\label{eq: 1. problem equation}
\end{equation}
where $(a_{ij})$ is a symmetric matrix-valued and Borel measurable
function on $\mathbb{R}^{n}$. Throughout the article, we always assume
that there exists a number $\lambda>0$ such that 
\begin{equation}
\lambda\vert\xi\vert^{2}\leq\sum_{i,j=1}^{n}a_{ij}\xi_{i}\xi_{j}\leq\frac{1}{\lambda}\vert\xi\vert^{2}\tag*{\textbf{E}}\label{eq: 1. ellipticity}
\end{equation}
for every $\xi\in\mathbb{R}^{n}$, and that $b=(b_{i})$ is a divergence-free
vector field, i.e.
\begin{equation}
\sum_{i=1}^{n}\partial_{x_{i}}b_{i}(t,x)=0\tag*{\textbf{S}}\label{eq: 1. divergence free}
\end{equation}
in the sense of distributions for all $t$. Here we only deal with
the case that $b$ belongs to Lebesgue spaces $L^{l}(0,T;L^{q}(\rn))$
(or $L_{t}^{l}L_{x}^{q}$ for short) for $l,q\in[1,\infty]$, and
we will denote 
\[
\Lambda:=\Vert b\Vert_{L_{t}^{l}L_{x}^{q}}=\left(\int_{0}^{T}\left(\int_{\rn}\vert b(t,x)\vert^{q}dx\right)^{\frac{l}{q}}dt\right)^{\frac{1}{l}}.
\]

Equation (\ref{eq: 1. problem equation}) has been well studied without
the divergence-free condition (\ref{eq: 1. divergence free}). A classical
monograph on such equation is \cite{ladyzhenskaia1988linear} by Ladyzhenskaia
et al. If $b$ is assumed to be in $L^{l}(0,T;L^{q}(\rn))$ with 
\[
\gamma=\frac{2}{l}+\frac{n}{q}\leq1,\qquad l\in[2,\infty)\mbox{ and }q\in(n,\infty],
\]
then there is a unique weak solution with Hölder regularity. If $\gamma<1$,
in \cite{aronson1968non}, Aronson proved that there exist Gaussian
upper and lower bounds for the fundamental solution, from which the
Hölder continuity of solutions can be deduced. We call such estimate
on the fundamental solution the Aronson estimate. The reason for such
conditions on $\gamma$ can be easily seen from the natural scaling
property of parabolic equations. Under the following scaling transformation
\[
u^{(\rho)}(t,x)=u(\rho^{2}t,\rho x),\quad a^{(\rho)}(t,x)=a(\rho^{2}t,\rho x),\quad b^{(\rho)}(t,x)=\rho b(\rho^{2}t,\rho x)
\]
for $\rho>0$, if $u$ is a solution to (\ref{eq: 1. problem equation})
with ellipticity constant $\lambda$, then $u^{(\rho)}$ is still
a solution to the parabolic equation with $(a^{(\rho)},b^{(\rho)})$
and condition (\ref{eq: 1. ellipticity}) is still satisfied for the
same $\lambda$. However, for the drift, we have $\Vert b^{(\rho)}\Vert_{L_{t}^{l}L_{x}^{q}}=\rho^{1-\gamma}\Vert b\Vert_{L_{t}^{l}L_{x}^{q}}$.
So when $\gamma=1$, it is scaling invariant and this is called the
critical case. If $\gamma<1$ (resp. $\gamma>1$), we call them subcritical
(resp. supercritical). Since the Harnack inequality has its constant
depending on $n,\lambda$ and $\Lambda$, in the supercritical case,
we are unable to obtain the Harnack inequality uniformly in small
scales. But in the critical and subcritical cases, we still can control
solutions in small scales to obtain the Harnack inequality, hence
obtain Hölder continuity. However, an exceptional case is $L^{\infty}(0,T;L^{n}(\rn))$,
which is critical, but the Harnack inequality fails. The simple reason
is that the energy estimate fails in this case.

In this article we will assume that $b$ is a divergence-free vector
field, which is significant for applications to the study of weak
solutions to incompressible fluid equations. There have been many
works concerning this problem. In \cite{osada1987diffusion}, assuming
that $b\in L_{t}^{\infty}(W_{x}^{-1,\infty})$, Osada established
the fundamental solution estimate following the idea of Nash \cite{nash1958continuity}.
Then in \cite{zhang2004strong}, Zhang obtained the exponential decay
upper bound for the fundamental solution when $m\in(1,2]$ and $b$
satisfies the following entropy condition

\begin{equation}
\int_{0}^{T}\int_{\rn}\vert b\vert^{m}\varphi^{2}dxdt\leq C\int_{0}^{T}\int_{\rn}\vert\nabla\varphi\vert^{2}dxdt\label{eq: zhang condition}
\end{equation}
for every smooth function $\varphi$ on $[0,T]\times\rn$ with compact
support in space. Such condition can be traced back to \cite{kovalenko1991c_0},
in which the previous entropy condition was first introduced for the
time independent case in order to construct a semigroup theory. The
Sobolev embedding allows us to deduce from the entropy condition (\ref{eq: zhang condition})
that $b\in L^{\infty}(0,T;L^{\frac{mn}{2}}(\rn))$. Therefore, the
entropy condition is effectively scaling invariant when $m=2$, i.e.
it is a critical case, while it is supercritical if $m\in(1,2)$.
For the critical case, in an interesting paper \cite{semenov2006regularity}
Semenov further developed a more general condition 
\begin{align}
\int_{0}^{T}\int_{\rn}\vert b\vert\varphi^{2}dxdt\leq & C_{1}\int_{0}^{T}\left[\int_{\rn}\vert\nabla\varphi\vert^{2}dx\right]^{\frac{1}{2}}\left[\int_{\rn}\vert\varphi\vert^{2}dx\right]^{\frac{1}{2}}dt\nonumber \\
 & +\int_{0}^{T}(C_{2}+C(t))\int_{\rn}\vert\varphi\vert^{2}dxdt\label{eq: seminov condition}
\end{align}
and obtained the existence and uniqueness of Hölder continuous weak
solutions for the parabolic equation. For the supercritical case,
Zhang \cite{zhang2006local} considered the following entropy condition:
\begin{equation}
\int_{0}^{T}\int_{\rn}\vert b\vert(\ln(1+\vert b\vert))^{2}\varphi^{2}dxdt\leq C\int_{0}^{T}\int_{\rn}\vert\nabla\varphi\vert^{2}dxdt\label{eq: zhang condition 2}
\end{equation}
and he proved the existence of a bounded weak solution in this case.
In fact, such weak solution must be Hölder continuous in space for
each fixed time $t$. In \cite{friedlander2011global,seregin2012divergence},
assuming that $b\in L_{t}^{\infty}(\textrm{BMO}_{x}^{-1})$ together
with some extra technical conditions, the Hölder continuity of weak
solution is obtained using Moser-De Giorgi's scheme. In a recent paper
\cite{qian2016parabolic}, only assuming that $b\in L_{t}^{\infty}(\textrm{BMO}_{x}^{-1})$,
the Aronson estimate is proved. Further, the uniqueness and Hölder
continuity of the weak solution is also proved. A diffusion process
to the same operator $L_{t}$ can be constructed using the Aronson
estimate. All these results are for divergence-free vector fields
$b$ which belong to the critical case. Under supercritical conditions,
recently in \cite{ignatova2016harnack}, assuming that $b\in L_{t,x}^{q}\cap L_{t}^{\infty}L_{x}^{2}$
with $q\in(\frac{n}{2}+1,n+2]$, Ignatova, Kukavica and Ryzhik proved
a weak Harnack inequality. The constant in the weak Harnack inequality
explodes as the radius of the parabolic ball goes to zero, hence it
fails to yield the Hölder continuity of weak solutions. 

\textcolor{black}{The strongest regularity result, in this respect,
is still the Aronson estimate, which was established in \cite{aronson1968non}.
The original proof of Aronson used the Harnack inequality. Later in
\cite{fabes1989new,norris1991estimates,stroock1988diffusion}, Stroock
et al. proved the Aronson estimate directly by using Nash's scheme
\cite{nash1958continuity}. Under the divergence-free condition, the
dual operator is of the same form up to a sign on $b$, so that Nash's
scheme demonstrates its full power as well for divergence-free vector
fields. Inspired by Stroock et al. \cite{fabes1989new,norris1991estimates,stroock1988diffusion},
our main result in this paper is the upper and lower bounds of fundamental
solutions for the critical and supercrtical cases $b\in L^{l}(0,T;L^{q}(\rn))$
with $\gamma=\frac{2}{l}+\frac{n}{q}\in[1,2)$ for divergence-free
drift $b$. }

In fact we will establish several heat kernel estimates for a range
of critical and supercritical conditions on divergence-free drifts
$b$. We will separate them into upper bound and lower bound. Since
the divergence-free condition on drift $b$ prevents the formation
of local blow up, we have the following upper bound which is of exponential
decay. 
\begin{thm}
\label{upperbound} Suppose conditions (\ref{eq: 1. ellipticity}),
(\ref{eq: 1. divergence free}) hold and $b\in L^{l}(0,T;L^{q}(\rn))$
for some $n\geq3$, $l>1$, $q>\frac{n}{2}$ such that $1\leq\gamma<2$.
In addition, we assume that $a$ and $b$ are smooth with bounded
derivatives. Let $\Gamma(t,x;\tau,\xi)$ be the fundamental solution
of (\ref{eq: 1. problem equation}). Then 
\begin{equation}
\Gamma(t,x;\tau,\xi)\leq\frac{C}{(t-\tau)^{n/2}}\exp\left(m(t-\tau,x-\xi)\right),\label{eq: 0 theorem 1 upper bound}
\end{equation}
where
\[
m(t,x)=\min_{\alpha\in\mathbb{R}^{n}}(C(\vert\alpha\vert^{2}t+\vert\alpha\vert^{\mu}\Lambda^{\mu}t^{\nu})+\alpha\cdot x)
\]
with $\Lambda=\Vert b\Vert_{L^{l}(0,T;L^{q}(\rn))}$, $\mu=\frac{2}{2-\gamma+\frac{2}{l}}$,
$\nu=\frac{2-\gamma}{2-\gamma+\frac{2}{l}}$ and $C=C(l,q,n,\lambda)$. 
\end{thm}

The point here is that the upper bound above only depends on $n$,
$\lambda$ and $\Lambda$, but not on the bounds of the derivatives
of $a$ or $b$. \textcolor{black}{One restriction of Nash's scheme
is that its iteration procedure requires the bound on $b$ to be uniform
in time, i.e. $l=\infty$. So in general, we use Moser's iteration
scheme instead and use cut-off functions.}

\textcolor{black}{When $\gamma=1$, the upper bound (\ref{eq: 0 theorem 1 upper bound})
is in fact a Gaussian function, which further implies a Gaussian lower
bound and regularity theory. Here we only write down the result in
the case when $b\in L^{\infty}(0,T;L^{n}(\rn))$, because it is the
marginal case for which the regularity result is missing if $b$ is
not assumed to be divergence-free. More explicitly, we prove the following}
\begin{thm}
\label{The-fundamental-solution-ln}The fundamental solution $\Gamma(t,x;\tau,\xi)$
to (\ref{eq: 1. problem equation}) satisfying conditions (\ref{eq: 1. ellipticity})
and (\ref{eq: 1. divergence free}), with $b\in L^{\infty}(0,T;L^{n}(\rn))$
possesses the following heat kernel estimate
\[
\frac{1}{C(t-\tau)^{n/2}}\exp\left(-\frac{C\vert x-\xi\vert^{2}}{t-\tau}\right)\leq\Gamma(t,x;\tau,\xi)\leq\frac{C}{(t-\tau)^{n/2}}\exp\left(-\frac{\vert x-\xi\vert^{2}}{C(t-\tau)}\right)
\]
for $t>\tau$, where $\Lambda=\Vert b\Vert_{L^{\infty}(0,T;L^{n}(\rn))}$
and $C$ depends only on $n\geq3$, $\lambda$ and $\varLambda$.
\end{thm}

This result shows that the divergence-free condition brings extra
regularity to weak solutions. For its discussion in other papers,
we refer to \cite{liskevich2004extra,semenov2006regularity,seregin2012divergence,zhang2004strong}.
Under the assumption that $b\in L^{\infty}(0,T;L^{n}(\rn))$, the
entropy condition (\ref{eq: seminov condition}) in \cite{semenov2006regularity}
is satisfied, and the above two-side Gaussian estimate can therefore
be deduced from Semenov's main results too. Here we give a simpler
and direct proof. 

In the supercritical case $\gamma\in(1,2)$, using upper bound (\ref{eq: 0 theorem 1 upper bound}),
we still can derive a lower bound. Actually, we know the fundamental
solution is conservative, and in fact, for fixed $t>\tau$, $\Gamma(t,x;\tau,\xi)$
is a probability density in $x$ (and in $\xi$ as well due to the
divergence-free condition). Because the upper bound decays exponentially,
we can find a radius $\tilde{R}(t)$ such that $\Gamma$ has a lower
bound inside the ball of radius $\tilde{R}(t)$. However, we can not
hope too much for the lower bound for supercritical case. Using current
techniques, we are able to establish the following theorem.
\begin{thm}
\label{lowerbound}Assume that $a$ and $b$ are smooth with bounded
derivatives. Suppose conditions (\ref{eq: 1. ellipticity}), (\ref{eq: 1. divergence free})
hold and $b\in L^{l}(0,T;L^{q}(\rn))$ for some $n\geq3$, $l\geq2$,
$q\geq2$ such that $1<\gamma<2$. For any $\kappa>0$, there is a
constant $C>0$ depending only on $\kappa,l,q,n,\lambda$ and $\Lambda=\Vert b\Vert_{L^{l}(0,T;L^{q}(\rn))}$
such that

\begin{equation}
\varGamma(t,x;0,\xi)\geq\exp\left[-Ct^{(\frac{n}{2}+1)(1-\gamma)}\left(\ln\frac{1}{t}\right)^{n+2}\right]\label{eq: 4.  general lowerbound step 2-1}
\end{equation}
for $x,\xi\in B(0,\kappa\tilde{R}(t))$ and small enough $t$, where
$\tilde{R}(s)=Cs^{(2-\gamma)/2}\ln\frac{1}{s}$ for $s>0$.
\end{thm}

Such form of lower bounds also appear in \cite{barlow2012equivalence},
but in a rather different setting of Dirichlet forms. Although we
only deal with the lower bound in the cone $B(0,R(t))$ for small
$t>0$, by the Chapman-Kolmogorov equation, we can extend this lower
bound to the whole space. So this form of lower bound is actually
the essence of lower bounds in heat kernel estimates, and it determines
the local behavior of solutions to the parabolic equation. In all
cases the upper bound looks stronger than the lower bounds, and the
lower bound in the supercritical case still fails to yield Hölder
continuity of weak solutions. Therefore, in the supercritical case,
the regularity problem for this kind of linear parabolic equations
remains an open problem.

The rest of this paper is organized as follows. Section 2 explains
how equation (\ref{eq: 1. problem equation}) is related to its dual
equation and we prove some lemmas which will be used later. In Section
3, we will prove an Aronson type estimate for fundamental solutions
under a class of critical and supercritical conditions on $b$. In
particular, in the case $L^{\infty}(0,T;L^{n}(\rn))$, the estimate
will be Gaussian type. In Section 4, as an application of the Aronson
estimate, we can prove the existence of a unique Hölder continuous
weak solution for the critical case $L^{\infty}(0,T;L^{n}(\rn))$
by an approximation approach. In an appendix, we prove Nash's continuity
theorem for the completeness of the paper because we will need this
result in section 4. 

\section{Several technical facts}

One important feature of the divergence-free condition is that the
adjoint equation of (\ref{eq: 1. problem equation}) essentially has
the same form up to a sign. Hence their fundamental solutions share
essentially the same property. Consider equation (\ref{eq: 1. problem equation})
on $[0,T]\times\rn$ and denote its fundamental solution as $\Gamma(t,x;\tau,\xi)$
with $0\leq\tau<t\leq T$, $x,\xi\in\rn$. For the adjoint equation
\begin{equation}
\partial_{t}u(t,x)-\sum_{i,j=1}^{n}\partial_{x_{i}}(a_{ij}(T-t,x)\partial_{x_{j}}u(t,x))+\sum_{i=1}^{n}b_{i}(T-t,x)\partial_{x_{i}}u(t,x)=0\label{eq: 1. adjoint problem equation}
\end{equation}
and its fundamental solution $\Gamma_{T}^{*}(t,x;\tau,\xi)$, we have
$\Gamma(t,x;\tau,\xi)=\Gamma_{T}^{*}(T-\tau,\xi;T-t,x)$. 

We will need the following elementary facts in the proof of Aronson
type estimates.\textcolor{brown}{{} }The first one is the Poincaré-Wirtinger
inequality for the Gaussian measures \cite[Corollary 1.7.3]{bogachev1998gaussian}.
In the sequel, we shall write $C_{b}^{1}(\rn)$ to be the space of
functions with bounded continuous first order derivative. 
\begin{lem}
\label{lem: poincare inequality}Let $\mu$ be the standard Gaussian
measure on $\rn$, i.e. $\mu(dx)=\mu(x)dx$ with $\mu(x)=\frac{1}{(2\pi)^{n/2}}\exp\left(-\frac{\vert x\vert^{2}}{2}\right)$.
Then for every $p\geq1$ 
\begin{equation}
\int_{\rn}\vert f(x)-\bar{f}\vert^{p}\mu(dx)\leq M(p)(\frac{\pi}{2})^{p}\int_{\rn}\vert\nabla f(x)\vert^{p}\mu(dx),\label{eq: 4. poincare inequality 1}
\end{equation}
for any $f\in C_{b}^{1}(\rn)$, where $\bar{f}=\int_{\rn}f(x)\mu(dx)$
and 
\[
M(p)=\int_{-\infty}^{\infty}\vert\xi\vert^{p}\frac{1}{(2\pi)^{1/2}}\exp\left(-\frac{1}{2}\vert\xi\vert^{2}\right)d\xi.
\]
Further, setting $\mu_{r}(x)=\frac{1}{r^{n/2}}\exp\left(-\frac{\pi\vert x\vert^{2}}{r}\right)$
and $\bar{f}_{r}=\int_{\rn}f(x)\mu_{r}(dx)$, one has
\begin{equation}
\int_{\rn}\vert f(x)-\bar{f}_{r}\vert^{p}\mu_{r}(dx)\leq M(p)(\frac{\pi}{2})^{p}(\frac{r}{2\pi})^{p/2}\int_{\rn}\vert\nabla f(x)\vert^{p}\mu_{r}(dx).\label{eq: 4. poincare inequality 2}
\end{equation}
\end{lem}

Actually, Lemma \ref{lem: poincare inequality} can be extended to
any function with weak derivative such that both sides of the Poincaré-Wirtinger
inequality are well defined using a truncation and approximation argument.
Also we will need the following lemma on a Riccati differential inequality.
\begin{lem}
\label{lem: 4. differential inequality}Suppose a non-positive valued
function $u$ is continuous and differentiable on $\left[\frac{T}{2},T\right]$,
where $T>0$ is a constant. If $u$ satisfies the Ricatti differential
inequality
\begin{equation}
u'(t)\geq-\alpha+\beta u(t)^{2}\label{eq:diffq1}
\end{equation}
for $t\in\left[\frac{T}{2},T\right]$, where $\alpha,\beta>0$ are
two constants, then 
\[
u(T)\geq\min\left\{ -\alpha T-2\sqrt{\frac{\alpha}{\beta}},-\frac{8}{3\beta T}\right\} .
\]
\end{lem}

\begin{proof}
If $u(T)\geq-\alpha T-2\sqrt{\frac{\alpha}{\beta}}$, then the proof
is done. Otherwise, integrating the differential inequality (\ref{eq:diffq1})
from $T/2$ to $T$, we have $u(T)-u(t)\geq-\frac{\alpha T}{2}$ for
any $t\in[\frac{T}{2},T]$. In other words, we have
\[
u(t)\leq u(T)+\frac{\alpha T}{2}\leq-\alpha T-2\sqrt{\frac{\alpha}{\beta}}+\frac{\alpha T}{2}
\]
which in turn yields that $u(t)\leq-2\sqrt{\frac{\alpha}{\beta}}$.
Notice that $u(t)$ is negative on $\left[\frac{T}{2},T\right]$ and
therefore $u(t)^{2}\geq4\frac{\alpha}{\beta}$. Hence differential
inequality (\ref{eq:diffq1}) implies that 
\[
u'(t)\geq\beta\left(-\frac{\alpha}{\beta}+u(t)^{2}\right)\geq\beta\left(-\frac{1}{4}u(t)^{2}+u(t)^{2}\right)=\frac{3\beta}{4}u(t)^{2}
\]
for every $t\in[\frac{T}{2},T]$. Dividing both sides by $u(t)^{2}$
and integrating from $t\in\left[\frac{T}{2},T\right]$ to $T$, we
obtain that
\[
\frac{1}{u(T)}\leq-\frac{3\beta T}{8}+\frac{1}{u(t)}\leq-\frac{3\beta T}{8}.
\]
In particular $u(T)\geq-\frac{8}{3\beta T}$ and the proof is complete.
\end{proof}
If $\alpha$ is a function depending on $t$, which is integrable
and non-negative, then we still can derive a lower bound on $u$. 
\begin{lem}
\label{lem: 2. generalized differential inequality}Let $T>0$. Suppose
a non-positive function $u$ is continuous on $[\frac{T}{2},T]$,
and satisfies the following integral inequality

\begin{equation}
u(t_{2})-u(t_{1})\geq\int_{t_{1}}^{t_{2}}\left(-\alpha(t)+\beta u(t)^{2}\right)dt,\qquad\mbox{for all }\frac{T}{2}\leq t_{1}<t_{2}\leq T,\label{eq:diffq2}
\end{equation}
where $\alpha$ is non-negative and integrable in $[\frac{T}{2},T]$,
and $\beta>0$ is a constant, then 
\[
u(T)\geq-\int_{\frac{T}{2}}^{T}\alpha(t)dt-C\beta^{-1}T^{-1}
\]
for some $C>0$.
\end{lem}

\begin{proof}
Let $C_{1}>C_{2}$ be a constant to be determined later, where $C_{2}=\int_{\frac{T}{2}}^{T}\alpha(t)dt$.
Suppose $u(T)<-C_{1}$. Then for any $t\in\left[\frac{T}{2},T\right]$
it holds that 
\[
u(t)\leq u(T)+\int_{t}^{T}\alpha(s)ds<-C_{1}+C_{2}=:-C_{3},
\]
where $C_{3}>0$ since $C_{1}>C_{2}$. So $u(t)$ is negative, which
implies that $u(t)^{2}\geq C_{3}^{2}$. Now we integrate (\ref{eq:diffq2})
to deduce that
\begin{equation}
u(t)\leq u(T)+\int_{t}^{T}\alpha(s)ds-\int_{t}^{T}\beta u(s)^{2}ds<-\int_{t}^{T}\beta u(s)^{2}ds,\label{eq: 4. inequality lemma 1}
\end{equation}
which implies that $u(t)\leq-\beta C_{3}^{2}(T-t)$ for all $t\in\left[\frac{T}{2},T\right]$.
Repeating the procedure of using the old bound of $u(t)$ and (\ref{eq: 4. inequality lemma 1})
to obtain a new bound, we deduce that
\[
u(t)\leq-\beta^{2^{m}-1}C_{3}^{2^{m}}(T-t)^{2^{m}-1}\prod_{k=1}^{m}(\frac{1}{2^{k}-1})^{2^{m-k}}\leq-C\left[\beta C_{3}(T-t)\prod_{k=1}^{m}(2^{-\frac{k}{2^{k}}})\right]^{2^{m}}
\]
after $m$ times. Since $\inf_{m}\prod_{k=1}^{m}(2^{-\frac{k}{2^{k}}})=C_{4}>0$,
the right-hand side can be arbitrarily small at time $\frac{T}{2}$
if $\beta C_{3}(T-\frac{T}{2})C_{4}>1$, which contradicts to the
fact that $u$ is finite. So if we take $C_{3}>\frac{2}{\beta TC_{4}}$,
i.e. $C_{1}=C_{2}+C\beta^{-1}T^{-1}$ for some constant $C$, then
$u(T)\geq-C_{1}$.
\end{proof}

\section{Aronson type estimates }

In this section, we will prove one of our main results, which is an
\emph{a priori} estimate of the fundamental solution to equation (\ref{eq: 1. problem equation}).
We will assume that $a\in C^{\infty}([0,T]\times\rn)$ and $b\in C^{\infty}([0,T],C_{0}^{\infty}(\rn))$
so that there exists a unique regular fundamental solution. 

\subsection{The upper bound}

The idea here is to estimate the $h$-transform of the fundamental
solution, which was first used by E. B. Davies \cite{davies1987explicit}.
But here we will use Moser's approach instead of Nash's to prove the
upper bound because it has the potential of applicability to more
general cases where $b\in L^{l}(0,T;L^{q}(\rn))$ satisfying
\begin{equation}
1\leq\frac{2}{l}+\frac{n}{q}<2\tag*{\textbf{A}}\label{eq: 3. Critical assumption}
\end{equation}
for $n\geq3$, $l>1$ and $q>\frac{n}{2}$. 

Given a function $\psi$ on $\rn$ which is smooth and has bounded
derivatives, we define the operator 
\begin{align*}
A_{t}^{\psi}u(x) & =\exp(-\psi(x))\sum_{i,j=1}^{n}\partial_{x_{i}}(a_{ij}(t,x)\partial_{x_{j}}[\exp(\psi(x))u(x)])\\
 & \quad-\exp(-\psi(x))\sum_{i=1}^{n}b_{i}(t,x)\partial_{x_{i}}[\exp(\psi(x))u(x)].
\end{align*}
Then its corresponding fundamental solution is 
\[
\Gamma^{\psi}(t,x;\tau,\xi)=\exp(-\psi(x))\Gamma(t,x;\tau,\xi)\exp(\psi(\xi)).
\]
For any $f\in C_{0}^{\infty}(\rn)$, we define a linear operator $\Gamma_{\tau,t}^{\psi}:C_{0}^{\infty}(\rn)\rightarrow L^{2}(\rn)$
as
\begin{align*}
\Gamma_{\tau,t}^{\psi}f(x) & =\int_{\rn}f(\xi)\Gamma^{\psi}(t,x;\tau,\xi)\,d\xi\\
 & =\int_{\rn}f(\xi)\exp(-\psi(x))\Gamma(t,x;\tau,\xi)\exp(\psi(\xi))\,d\xi.
\end{align*}
It is easy to observe that the adjoint operator of $\Gamma_{\tau,t}^{\psi}$
can be identified as the following linear operator 
\[
\Gamma_{\tau,t}^{\psi\perp}f(x)=\int_{\rn}f(\xi)\exp(-\psi(\xi))\Gamma(\tau,\xi;t,x)\exp(\psi(x))\,d\xi,
\]
and they satisfy 
\begin{equation}
\langle\Gamma_{\tau,t}^{\psi}f,g\rangle_{L^{2}(\rn)}=\langle f,\Gamma_{\tau,t}^{\psi\perp}g\rangle_{L^{2}(\rn)}.\label{eq: 3. dual equation}
\end{equation}
\begin{lem}
\label{lem5}Suppose $(a,b)$ satisfies conditions (\ref{eq: 1. ellipticity}),
(\ref{eq: 1. divergence free}) and (\ref{eq: 3. Critical assumption}).
Given $\alpha\in\rn$, and $\psi(x)=\alpha\cdot x$, set
\[
f_{t}(x)=\Gamma_{0,t}^{\psi}f(x)=\int_{\rn}f(\xi)\Gamma^{\psi}(t,x;0,\xi)\,d\xi
\]
for $f\in C_{0}^{\infty}(\rn)$. Then there exists a constant $C$
depending on $(n,l,q)$ such that 
\begin{equation}
\Vert f_{t}\Vert_{L_{x}^{2}}^{2}\leq\exp\left(\frac{2\vert\alpha\vert^{2}}{\lambda}t+2C\lambda^{-\frac{1+\theta}{1-\theta}}\vert\alpha\vert^{\frac{2}{1-\theta}}\Lambda^{\mu}t^{\nu}\right)\cdot\Vert f\Vert_{L_{x}^{2}}^{2},\label{eq: energy estimate -1}
\end{equation}
where $\theta=\frac{n}{q}-1$, $\mu=\frac{2}{2-\gamma+\frac{2}{l}}$,
$\nu=\frac{2-\gamma}{2-\gamma+\frac{2}{l}}$, $\gamma=\frac{2}{l}+\frac{n}{q}$
and $\Lambda=\Vert b\Vert_{L^{l}(0,T;L^{q}(\rn))}$. 
\end{lem}

\begin{proof}
We begin with the fact that $f_{t}$ satisfies 
\[
\frac{d}{dt}\Vert f_{t}\Vert_{L_{x}^{2}}^{2}=2\langle A_{t}^{\psi}f_{t},f_{t}\rangle_{L^{2}(\mathbb{R}^{n})}.
\]
It follows that
\begin{align*}
 & \frac{1}{2}\left(\Vert f_{t}\Vert_{L_{x}^{2}}^{2}-\Vert f\Vert_{L_{x}^{2}}^{2}\right)=\frac{1}{2}\int_{0}^{t}\frac{d}{ds}\Vert f_{s}\Vert_{L_{x}^{2}}^{2}\;ds=\int_{0}^{t}\int_{\mathbb{R}^{n}}A_{s}^{\psi}f_{s}(x)\cdot f_{s}(x)\;dxds\\
 & \quad=-\int_{0}^{t}\int_{\mathbb{R}^{n}}\sum_{i,j=1}^{n}a_{ij}(s,x)\partial_{x_{j}}[\exp(\psi(x))f_{s}(x)]\partial_{x_{i}}[\exp(-\psi(x))f_{s}(x)]\;dxds\\
 & \quad\quad-\int_{0}^{t}\int_{\mathbb{R}^{n}}\sum_{i=1}^{n}b_{i}(s,x)\partial_{x_{i}}[\exp(\psi(x))f_{s}(x)][\exp(-\psi(x))f_{s}(x)]\;dxds\\
 & \quad=\int_{0}^{t}\int_{\mathbb{R}^{n}}\langle\alpha\cdot a(s,x),\alpha\rangle f_{s}^{2}(x)\;dxds-\int_{0}^{t}\int_{\mathbb{R}^{n}}\langle\nabla f_{s}(x)\cdot a(s,x),\nabla f_{s}(x)\rangle\;dxds\\
 & \quad\quad-\int_{0}^{t}\int_{\mathbb{R}^{n}}\langle\alpha\cdot a(s,x),\nabla f_{s}(x)\rangle f_{s}(x)\;dxds+\int_{0}^{t}\int_{\mathbb{R}^{n}}\langle\nabla f_{s}(x)\cdot a(s,x),\alpha\rangle f_{s}(x)\;dxds\\
 & \quad\quad-\int_{0}^{t}\int_{\mathbb{R}^{n}}\langle b(s,x),\alpha\rangle f_{s}^{2}(x)\;dxds-\int_{0}^{t}\int_{\mathbb{R}^{n}}\langle b(s,x),\nabla f_{s}(x)\rangle f_{s}(x)\;dxds.
\end{align*}
Since $b$ is divergence-free, we have for any $s$ that 
\[
\int_{\mathbb{R}^{n}}\langle b(s,x),\nabla f_{s}(x)\rangle f_{s}(x)\;dx=0.
\]
The third and fourth terms cancel each other and condition (\ref{eq: 1. ellipticity})
gives 
\[
\begin{aligned} & \frac{1}{2}\left(\Vert f_{t}\Vert_{L_{x}^{2}}^{2}-\Vert f\Vert_{L_{x}^{2}}^{2}\right)\\
 & \quad=\int_{0}^{t}\int_{\mathbb{R}^{n}}\langle\alpha\cdot a(s,x),\alpha\rangle f_{s}^{2}(x)\;dxds-\int_{0}^{t}\int_{\mathbb{R}^{n}}\langle\nabla f_{s}(x)\cdot a(s,x),\nabla f_{s}(x)\rangle\;dxds\\
 & \quad\quad-\int_{0}^{t}\int_{\mathbb{R}^{n}}\langle b(s,x),\alpha\rangle f_{s}^{2}(x)\;dxds\\
 & \quad\leq\int_{0}^{t}\frac{\vert\alpha\vert^{2}}{\lambda}\Vert f_{s}\Vert_{L_{x}^{2}}^{2}\;ds-\int_{0}^{t}\lambda\Vert\nabla f_{s}\Vert_{L_{x}^{2}}^{2}\;ds-\int_{0}^{t}\int_{\mathbb{R}^{n}}\langle b(s,x),\alpha\rangle f_{s}^{2}(x)\;dxds.
\end{aligned}
\]
For the last term, one obtains the following estimate

\[
\begin{aligned}\left\vert \int_{0}^{t}\int_{\mathbb{R}^{n}}\langle b(s,x),\alpha\rangle f_{s}^{2}(x)\;dxds\right\vert  & \leq\int_{0}^{t}\vert\alpha\vert\Vert b(s,\cdot)\Vert_{L_{x}^{q}}\Vert f_{s}^{1+\theta}\Vert_{L_{x}^{r_{1}}}\Vert f_{s}^{1-\theta}\Vert_{L_{x}^{r_{2}}}\;ds\\
 & =\int_{0}^{t}\vert\alpha\vert\Vert b(s,\cdot)\Vert_{L_{x}^{q}}\Vert f_{s}\Vert_{L_{x}^{(1+\theta)r_{1}}}^{1+\theta}\Vert f_{s}\Vert_{L_{x}^{(1-\theta)r_{2}}}^{1-\theta}\;ds,
\end{aligned}
\]
where 
\[
\theta=\frac{n}{q}-1,\qquad(1+\theta)r_{1}=\frac{2n}{n-2},\qquad(1-\theta)r_{2}=2.
\]
By Sobolev's embedding and Young's inequality, we can further control
it as follows

\[
\begin{aligned} & \left\vert \int_{0}^{t}\int_{\mathbb{R}^{n}}\langle b(s,x),\alpha\rangle f_{s}^{2}(x)\;dxds\right\vert \\
 & \quad\leq\int_{0}^{t}C\vert\alpha\vert\Vert b(s,\cdot)\Vert_{L_{x}^{q}}\Vert f_{s}\Vert_{L_{x}^{2}}^{1-\theta}\Vert\nabla f_{s}\Vert_{L_{x}^{2}}^{1+\theta}\;ds\\
 & \quad=\int_{0}^{t}((\frac{2}{\lambda})^{\frac{1+\theta}{2}}C\vert\alpha\vert\Vert b(s,\cdot)\Vert_{L_{x}^{q}}\Vert f_{s}\Vert_{L_{x}^{2}}^{1-\theta})((\frac{\lambda}{2})^{\frac{1+\theta}{2}}\Vert\nabla f_{s}\Vert_{L_{x}^{2}}^{1+\theta})\;ds\\
 & \quad\leq\int_{0}^{t}\frac{1-\theta}{2}(\frac{2}{\lambda})^{\frac{1+\theta}{1-\theta}}(C\vert\alpha\vert\Vert b(s,\cdot)\Vert_{L_{x}^{q}})^{\frac{2}{1-\theta}}\Vert f_{s}\Vert_{L_{x}^{2}}^{2}+\frac{1+\theta}{2}\frac{\lambda}{2}\Vert\nabla f_{s}\Vert_{L_{x}^{2}}^{2}\;ds\\
 & \quad\leq\int_{0}^{t}C(\frac{1}{\lambda})^{\frac{1+\theta}{1-\theta}}(\vert\alpha\vert\Vert b(s,\cdot)\Vert_{L_{x}^{q}})^{\frac{2}{1-\theta}}\Vert f_{s}\Vert_{L_{x}^{2}}^{2}+\frac{\lambda}{2}\Vert\nabla f_{s}\Vert_{L_{x}^{2}}^{2}\;ds.
\end{aligned}
\]
Combining all the estimates above, one has
\begin{align*}
\Vert f_{t}\Vert_{L_{x}^{2}}^{2} & \leq\Vert f\Vert_{L_{x}^{2}}^{2}+2\int_{0}^{t}\frac{\vert\alpha\vert^{2}}{\lambda}\Vert f_{s}\Vert_{L_{x}^{2}}^{2}-\lambda\Vert\nabla f_{s}\Vert_{L_{x}^{2}}^{2}\;ds\\
 & \quad+\int_{0}^{t}C(\frac{1}{\lambda})^{\frac{1+\theta}{1-\theta}}(\vert\alpha\vert\Vert b(s,\cdot)\Vert_{L_{x}^{q}})^{\frac{2}{1-\theta}}\Vert f_{s}\Vert_{L_{x}^{2}}^{2}+\frac{\lambda}{2}\Vert\nabla f_{s}\Vert_{L_{x}^{2}}^{2}\;ds\\
 & \leq\Vert f\Vert_{L_{x}^{2}}^{2}+2\int_{0}^{t}\left(\frac{\vert\alpha\vert^{2}}{\lambda}+C(\frac{1}{\lambda})^{\frac{1+\theta}{1-\theta}}(\vert\alpha\vert\Vert b(s,\cdot)\Vert_{L_{x}^{q}})^{\frac{2}{1-\theta}}\right)\cdot\Vert f_{s}\Vert_{L_{x}^{2}}^{2}\;ds.
\end{align*}
Recall $\frac{2}{l}+\frac{n}{q}=\gamma$ with $1\leq\gamma<2$ and
$\Lambda=\Vert b\Vert_{L^{l}([0,T],L^{q}(\mathbb{R}^{n}))}$. Hölder's
inequality implies that
\begin{align*}
\int_{0}^{t}\Vert b(s,\cdot)\Vert_{L_{x}^{q}}^{\frac{2}{1-\theta}}\;ds & =\int_{0}^{t}\Vert b(s,\cdot)\Vert_{L_{x}^{q}}^{\frac{2}{(2-\gamma+\frac{2}{l})}}\;ds\leq\left(\int_{0}^{t}\Vert b(s,\cdot)\Vert_{L_{x}^{q}}^{l}\;ds\right)^{\frac{\frac{2}{l}}{2-\gamma+\frac{2}{l}}}t^{\frac{2-\gamma}{2-\gamma+\frac{2}{l}}}\\
 & =\Lambda^{\frac{2}{2-\gamma+\frac{2}{l}}}t^{\frac{2-\gamma}{2-\gamma+\frac{2}{l}}},
\end{align*}
where we set $\frac{c}{l}=0$ if $l=\infty$. For simplicity, we denote
$\mu=\frac{2}{2-\gamma+\frac{2}{l}}$ and $\nu=\frac{2-\gamma}{2-\gamma+\frac{2}{l}}$.
Hence, by Grönwall's inequality and
\[
\Vert f_{t}\Vert_{L_{x}^{2}}^{2}\leq\Vert f\Vert_{L_{x}^{2}}^{2}+2\int_{0}^{t}\left(\frac{\vert\alpha\vert^{2}}{\lambda}+C(\frac{1}{\lambda})^{\frac{1+\theta}{1-\theta}}(\vert\alpha\vert\Vert b(s,\cdot)\Vert_{L_{x}^{q}})^{\frac{2}{1-\theta}}\right)\Vert f_{s}\Vert_{L_{x}^{2}}^{2}\;ds,
\]
we deduce that
\[
\begin{aligned}\Vert f_{t}\Vert_{L_{x}^{2}}^{2} & \leq\exp{\left(2\int_{0}^{t}\left(\frac{\vert\alpha\vert^{2}}{\lambda}+C(\frac{1}{\lambda})^{\frac{1+\theta}{1-\theta}}(\vert\alpha\vert\Vert b(s,\cdot)\Vert_{L_{x}^{q}})^{\frac{2}{1-\theta}}\right)\;ds\right)}\Vert f\Vert_{L_{x}^{2}}^{2}\\
 & \leq\exp{\left(\frac{2\vert\alpha\vert^{2}}{\lambda}t+2C(\frac{1}{\lambda})^{\frac{1+\theta}{1-\theta}}\vert\alpha\vert^{\frac{2}{1-\theta}}\Lambda^{\mu}t^{\nu}\right)}\Vert f\Vert_{L_{x}^{2}}^{2}.
\end{aligned}
\]
Now the proof is complete.
\end{proof}
\begin{lem}
\label{lem6}\textcolor{black}{Suppose that $(a,b)$, $\psi$ and
$f_{t}$ are defined as in Lemma \ref{lem5}. Fo}r any $p\geq1$ and
any smooth non-negative function $\eta$ on $[0,T]$ satisfying $\eta(0)=0$,
we have 
\begin{align*}
\Vert f_{t}^{p}\eta^{\sigma}\Vert_{L_{t}^{2\chi}L_{x}^{2\chi}}^{2} & \leq C\vert\alpha\vert^{2}p^{2}\Vert f_{t}^{p}\eta^{\sigma}\Vert_{L_{t}^{2}L_{x}^{2}}^{2}+C(\vert\alpha\vert p)^{\frac{2}{2-\gamma}}\Vert b\Vert_{L_{t}^{l}L_{x}^{q}}^{\frac{2}{2-\gamma}}\Vert f_{t}^{p}\eta^{^{\frac{1}{2-\gamma}}}\Vert_{L_{t}^{2}L_{x}^{2}}^{2}\\
 & \quad+C\int_{0}^{T}\int_{\mathbb{R}^{n}}\sigma f_{t}^{2p}(x)\vert\partial_{t}\eta(t)\vert\eta^{2\sigma-1}(t)\;dxdt.
\end{align*}
where $\chi=\frac{n+2}{n}$, $\sigma=\frac{1}{2-\gamma}$ and $C>0$
is a constant depending only on $l,q,n,\lambda$.
\end{lem}

\begin{proof}
For any $p\geq1$, we have 
\[
\frac{d}{dt}\Vert f_{t}\Vert_{L_{x}^{2p}}^{2p}=2p\langle A_{t}^{\psi}f_{t},f_{t}^{2p-1}\rangle_{L^{2}(\mathbb{R}^{n})}.
\]
Next we multiply both sides by $\eta^{2\sigma}$ and integrate on
$[0,T]$ to obtain
\begin{align}
 & \int_{0}^{T}\eta^{2\sigma}(t)\int_{\mathbb{R}^{n}}\partial_{t}f_{t}(x)f_{t}(x)^{2p-1}\;dxdt\nonumber \\
 & \quad=-\int_{0}^{T}\eta^{2\sigma}(t)\int_{\mathbb{R}^{n}}\langle\nabla(\exp(\psi(x))f_{t}(x))\cdot a(t,x),\nabla(\exp(-\psi(x))f_{t}^{2p-1}(x))\rangle\;dxdt\nonumber \\
 & \quad\quad-\int_{0}^{T}\eta^{2\sigma}(t)\int_{\mathbb{R}^{n}}\langle b(t,x),\nabla(\exp(\psi(x))f_{t}(x))\rangle\exp(-\psi(x))f_{t}^{2p-1}(x)\;dxdt\nonumber \\
 & \quad=\int_{0}^{T}\eta^{2\sigma}(t)\int_{\mathbb{R}^{n}}\langle\alpha\cdot a(t,x),\alpha\rangle f_{t}^{2p}(x)\;dxdt\nonumber \\
 & \quad\quad-(2p-1)\int_{0}^{T}\eta^{2\sigma}(t)\int_{\mathbb{R}^{n}}\langle\nabla f_{t}(x)\cdot a(t,x),\nabla f_{t}(x)\rangle f_{t}^{2p-2}(x)\;dxdt\nonumber \\
 & \quad\quad-(2p-1)\int_{0}^{T}\eta^{2\sigma}(t)\int_{\mathbb{R}^{n}}\langle\alpha\cdot a(t,x),\nabla f_{t}(x)\rangle f_{t}^{2p-1}(x)\;dxdt\nonumber \\
 & \quad\quad+\int_{0}^{T}\eta^{2\sigma}(t)\int_{\mathbb{R}^{n}}\langle\nabla f_{t}(x)\cdot a(t,x),\alpha\rangle f_{t}^{2p-1}(x)\;dxdt\nonumber \\
 & \quad\quad-\int_{0}^{T}\eta^{2\sigma}(t)\int_{\mathbb{R}^{n}}\langle b(t,x),\alpha\rangle f_{t}^{2p}(x)\;dxdt\nonumber \\
 & \quad\quad-\int_{0}^{T}\eta^{2\sigma}(t)\int_{\mathbb{R}^{n}}\langle b(t,x),\nabla f_{t}(x)\rangle f_{t}^{2p-1}(x)\;dxdt.\label{eq:local equation}
\end{align}
Condition (\ref{eq: 1. divergence free}) implies that
\[
\int_{\mathbb{R}^{n}}\langle b(t,x),\nabla f_{t}(x)\rangle f_{t}^{2p-1}(x)\;dx=0
\]
for any $t$, and hence the last term vanishes. Set $g_{t}=f_{t}^{p}$
for simplicity, then the left-hand side becomes 
\[
\begin{aligned} & \int_{0}^{T}\eta^{2\sigma}(t)\int_{\mathbb{R}^{n}}\partial_{t}f_{t}(x)f_{t}(x)^{2p-1}\;dxdt=\int_{0}^{T}\int_{\mathbb{R}^{n}}\eta^{2\sigma}(t)\frac{1}{2p}\partial_{t}(g_{t}^{2}(x))\;dxdt\\
 & \quad=\left.\int_{\mathbb{R}^{n}}\frac{1}{2p}\eta^{2\sigma}(t)g_{t}^{2}(x)\;dx\right\vert _{0}^{T}-\int_{0}^{T}\int_{\mathbb{R}^{n}}\frac{\sigma}{p}g_{t}^{2}(x)(\partial_{t}\eta(t))\eta^{2\sigma-1}(t)\;dxdt.
\end{aligned}
\]
Multiplying by $p$ on both sides of equation (\ref{eq:local equation}),
we obtain
\begin{align*}
 & \left.\int_{\mathbb{R}^{n}}\frac{1}{2}\eta^{2\sigma}(t)g_{t}^{2}(x)\;dx\right\vert _{0}^{T}-\int_{0}^{T}\int_{\mathbb{R}^{n}}\sigma g_{t}^{2}(x)(\partial_{t}\eta(t))\eta^{2\sigma-1}(t)\;dxdt\\
 & \quad=p\int_{0}^{T}\eta^{2\sigma}(t)\int_{\mathbb{R}^{n}}\langle\alpha\cdot a(t,x),\alpha\rangle g_{t}^{2}(x)\;dxdt\\
 & \quad\quad-\frac{(2p-1)}{p}\int_{0}^{T}\eta^{2\sigma}(t)\int_{\mathbb{R}^{n}}\langle\nabla g_{t}(x)\cdot a(t,x),\nabla g_{t}(x)\rangle\;dxdt\\
 & \quad\quad-(2p-2)\int_{0}^{T}\eta^{2\sigma}(t)\int_{\mathbb{R}^{n}}\langle\alpha\cdot a(t,x),\nabla g_{t}(x)\rangle g_{t}(x)\;dxdt\\
 & \quad\quad-p\int_{0}^{T}\eta^{2\sigma}(t)\int_{\mathbb{R}^{n}}\langle b(t,x),\alpha\rangle g_{t}^{2}(x)\;dxdt\\
 & \quad=I_{1}-I_{2}-I_{3}-I_{4}.
\end{align*}
Now we estimate each term individually as follows
\[
I_{1}\leq\int_{0}^{T}\eta^{2\sigma}(t)\frac{\vert\alpha\vert^{2}}{\lambda}p\Vert g_{t}\Vert_{L_{x}^{2}}^{2}\;dt,
\]
\[
-I_{2}-I_{3}\leq\int_{0}^{T}\eta^{2\sigma}(t)\frac{\vert\alpha\vert^{2}}{\lambda}(p-1)p\Vert g_{t}\Vert_{L_{x}^{2}}^{2}\;dt-\int_{0}^{T}\eta^{2\sigma}(t)\lambda\Vert\nabla g_{t}\Vert_{L_{x}^{2}}^{2}\;dt,
\]
\[
\begin{aligned}\vert I_{4}\vert & =\left\vert p\int_{0}^{T}\eta^{2\sigma}(t)\int_{\mathbb{R}^{n}}\langle b(t,x),\alpha\rangle g_{t}^{2}(x)\;dxdt\right\vert \\
 & \leq\int_{0}^{T}\int_{\mathbb{R}^{n}}p\vert b(t,x)\vert\vert g_{t}\eta^{\sigma}\vert^{\gamma}\vert g_{t}\vert^{2-\gamma}(\vert\alpha\vert\eta)\;dxdt\\
 & \leq\vert\alpha\vert p\Vert b\Vert_{L_{t}^{l}L_{x}^{q}}\Vert g_{t}\eta^{\sigma}\Vert_{L_{t}^{s}L_{x}^{r}}^{\gamma}\Vert g_{t}\eta^{^{\frac{1}{2-\gamma}}}\Vert_{L_{t}^{2}L_{x}^{2}}^{2-\gamma}
\end{aligned}
\]
since $\sigma\gamma=2\sigma-1$ and 
\[
\frac{1}{l}+\frac{\gamma}{s}+\frac{2-\gamma}{2}=1,\qquad\frac{1}{q}+\frac{\gamma}{r}+\frac{2-\gamma}{2}=1.
\]
From this relation, it is easy to see 
\[
\frac{2}{s}+\frac{n}{r}=\frac{n}{2},
\]
which yields the interpolation inequality 
\[
\Vert f\Vert_{L_{t}^{s}L_{x}^{r}}\leq C\Vert f\Vert_{L_{t}^{\infty}L_{x}^{2}}^{1-\beta}\Vert\nabla f\Vert_{L_{t}^{2}L_{x}^{2}}^{\beta},\qquad\beta=\frac{n}{2}-\frac{n}{r}.
\]
Together with Young's inequality, we deduce the following estimate
\begin{equation}
\Vert f\Vert_{L_{t}^{s}L_{x}^{r}}\leq C_{1}\Vert f\Vert_{L_{t}^{\infty}L_{x}^{2}}+C_{2}\Vert\nabla f\Vert_{L_{t}^{2}L_{x}^{2}}.\label{eq: interpolation inequality}
\end{equation}
Now we choose $\epsilon>0$ small enough such that, by Young's inequality,
we have 
\begin{align*}
\vert I_{4}\vert & \leq\epsilon\Vert g_{t}\eta^{\sigma}\Vert_{L_{t}^{s}L_{x}^{r}}^{2}+C(\epsilon)(\vert\alpha\vert p)^{\frac{2}{2-\gamma}}\Vert b\Vert_{L_{t}^{l}L_{x}^{q}}^{\frac{2}{2-\gamma}}\Vert g_{t}\eta^{^{\frac{1}{2-\gamma}}}\Vert_{L_{t}^{2}L_{x}^{2}}^{2}\\
 & \leq\frac{\lambda\wedge1}{4}(\Vert g_{t}\eta^{\sigma}\Vert_{L_{t}^{\infty}L_{x}^{2}}^{2}+\Vert\nabla g_{t}\eta^{\sigma}\Vert_{L_{t}^{2}L_{x}^{2}}^{2})+C(\vert\alpha\vert p)^{\frac{2}{2-\gamma}}\Vert b\Vert_{L_{t}^{l}L_{x}^{q}}^{\frac{2}{2-\gamma}}\Vert g_{t}\eta^{^{\frac{1}{2-\gamma}}}\Vert_{L_{t}^{2}L_{x}^{2}}^{2}.
\end{align*}
Combining these together, we conclude that
\begin{align*}
 & \left.\int_{\mathbb{R}^{n}}\frac{1}{2}\eta^{2\sigma}(t)g_{t}^{2}(x)\;dx\right\vert _{0}^{T}-\int_{0}^{T}\int_{\mathbb{R}^{n}}\sigma g_{t}^{2}(x)(\partial_{t}\eta(t))\eta^{2\sigma-1}(t)\;dxdt\\
 & \quad\leq\int_{0}^{T}\eta^{2\sigma}(t)\frac{\vert\alpha\vert^{2}}{\lambda}p^{2}\Vert g_{t}\Vert_{L_{x}^{2}}^{2}\;dt-\int_{0}^{T}\eta^{2\sigma}(t)\lambda\Vert\nabla g_{t}\Vert_{L_{x}^{2}}^{2}\;dt\\
 & \quad\quad+\frac{\lambda\wedge1}{4}(\Vert g_{t}\eta^{\sigma}\Vert_{L_{t}^{\infty}L_{x}^{2}}^{2}+\Vert\nabla g_{t}\eta^{\sigma}\Vert_{L_{t}^{2}L_{x}^{2}}^{2})+C(\vert\alpha\vert p)^{\frac{2}{2-\gamma}}\Vert b\Vert_{L_{t}^{l}L_{x}^{q}}^{\frac{2}{2-\gamma}}\Vert g_{t}\eta^{^{\frac{1}{2-\gamma}}}\Vert_{L_{t}^{2}L_{x}^{2}}^{2}.
\end{align*}
If we set $\eta(0)=0$, then the inequality above implies that
\begin{align*}
 & \frac{1}{2}\Vert g_{T}\eta^{\sigma}(T)\Vert_{L_{x}^{2}}^{2}+\frac{\lambda}{2}\Vert\nabla g_{t}\eta^{\sigma}\Vert_{L_{t}^{2}L_{x}^{2}}^{2}\\
 & \quad\leq\frac{\vert\alpha\vert^{2}p^{2}}{\lambda}\Vert g_{t}\eta^{\sigma}\Vert_{L_{t}^{2}L_{x}^{2}}^{2}+\frac{1}{4}\Vert g_{t}\eta^{\sigma}\Vert_{L_{t}^{\infty}L_{x}^{2}}^{2}+C(\vert\alpha\vert p)^{\frac{2}{2-\gamma}}\Vert b\Vert_{L_{t}^{l}L_{x}^{q}}^{\frac{2}{2-\gamma}}\Vert g_{t}\eta^{^{\frac{1}{2-\gamma}}}\Vert_{L_{t}^{2}L_{x}^{2}}^{2}\\
 & \quad\quad+\int_{0}^{T}\int_{\mathbb{R}^{n}}\sigma g_{t}^{2}(x)\vert\partial_{t}\eta(t)\vert\eta^{2\sigma-1}(t)\;dxdt,
\end{align*}
and the same is true if we replace $T$ by any $t\in[0,T]$. Hence
\begin{align*}
 & \frac{1}{4}\Vert g_{t}\eta^{\sigma}\Vert_{L_{t}^{\infty}L_{x}^{2}}^{2}+\frac{\lambda}{2}\Vert\nabla g_{t}\eta^{\sigma}\Vert_{L_{t}^{2}L_{x}^{2}}^{2}\\
 & \quad\leq\frac{\vert\alpha\vert^{2}p^{2}}{\lambda}\Vert g_{t}\eta^{\sigma}\Vert_{L_{t}^{2}L_{x}^{2}}^{2}+C(\vert\alpha\vert p)^{\frac{2}{2-\gamma}}\Vert b\Vert_{L_{t}^{l}L_{x}^{q}}^{\frac{2}{2-\gamma}}\Vert g_{t}\eta^{^{\frac{1}{2-\gamma}}}\Vert_{L_{t}^{2}L_{x}^{2}}^{2}\\
 & \quad\quad+\int_{0}^{T}\int_{\mathbb{R}^{n}}\sigma g_{t}^{2}(x)\vert\partial_{t}\eta(t)\vert\eta^{2\sigma-1}(t)\;dxdt.
\end{align*}
Applying the interpolation inequality (\ref{eq: interpolation inequality})
with $s=r=\chi=\frac{n+2}{n}$, we deduce that
\begin{align*}
\Vert g_{t}\eta^{\sigma}\Vert_{L_{t}^{2\chi}L_{x}^{2\chi}}^{2} & \leq C\vert\alpha\vert^{2}p^{2}\Vert g_{t}\eta^{\sigma}\Vert_{L_{t}^{2}L_{x}^{2}}^{2}+C(\vert\alpha\vert p)^{\frac{2}{2-\gamma}}\Vert b\Vert_{L_{t}^{l}L_{x}^{q}}^{\frac{2}{2-\gamma}}\Vert g_{t}\eta^{^{\frac{1}{2-\gamma}}}\Vert_{L_{t}^{2}L_{x}^{2}}^{2}\\
 & \quad+C\int_{0}^{T}\int_{\mathbb{R}^{n}}\sigma g_{t}^{2}(x)\vert\partial_{t}\eta(t)\vert\eta^{2\sigma-1}(t)\;dxdt,
\end{align*}
and the proof is complete.
\end{proof}
Now we can use the Moser's iteration to prove Theorem \ref{upperbound}.
\begin{proof}
[Proof of Theorem \ref{upperbound}]Define open intervals $I_{k}=((\frac{1}{2}-\frac{1}{2^{k+1}})T,T)$
and choose $\eta_{k}$ as cut-off functions such that $\eta_{k}=1$
on $I_{k}$, $\eta_{k}=0$ on $I_{0}\backslash I_{k-1}$ and $\vert\partial_{t}\eta_{k}\vert\leq4^{k}T^{-1}$.
Denote $L_{I_{k}\times\mathbb{R}^{n}}^{p}$ the $L^{p}$ space on
the space-time domain $I_{k}\times\rn$. Then 
\begin{align*}
 & \Vert g_{t}\Vert_{L_{I_{k}\times\mathbb{R}^{n}}^{2\chi}}^{2}\leq\Vert g_{t}\eta_{k}^{\sigma}\Vert_{L_{t}^{2\chi}L_{x}^{2\chi}}^{2}\\
 & \quad\leq C\vert\alpha\vert^{2}p^{2}\Vert g_{t}\eta_{k}^{\sigma}\Vert_{L_{t}^{2}L_{x}^{2}}^{2}+C(\vert\alpha\vert p)^{\frac{2}{2-\gamma}}\Vert b\Vert_{L_{t}^{l}L_{x}^{q}}^{\frac{2}{2-\gamma}}\Vert g_{t}\eta_{k}^{^{\frac{1}{2-\gamma}}}\Vert_{L_{t}^{2}L_{x}^{2}}^{2}\\
 & \quad\quad+C\int_{0}^{T}\int_{\mathbb{R}^{n}}\sigma g_{t}^{2}(x)\vert\partial_{t}\eta_{k}(t)\vert\eta_{k}^{2\sigma-1}(t)\;dxdt\\
 & \quad\leq C\vert\alpha\vert^{2}p^{2}\Vert g_{t}\Vert_{L_{I_{k-1}\times\mathbb{R}^{n}}^{2}}^{2}+C(\vert\alpha\vert p)^{\frac{2}{2-\gamma}}\Vert b\Vert_{L_{t}^{l}L_{x}^{q}}^{\frac{2}{2-\gamma}}\Vert g_{t}\Vert_{L_{I_{k-1}\times\mathbb{R}^{n}}^{2}}^{2}+C\sigma\frac{4^{k}}{T}\Vert g_{t}\Vert_{L_{I_{k-1}\times\mathbb{R}^{n}}^{2}}^{2}\\
 & \quad\leq C\left(\vert\alpha\vert^{2}p^{2}+p^{\frac{2}{2-\gamma}}\Vert b\Vert_{L_{t}^{l}L_{x}^{q}}^{\frac{2}{2-\gamma}}\vert\alpha\vert^{\frac{2}{2-\gamma}}+\sigma\frac{4^{k}}{T}\right)\Vert g_{t}\Vert_{L_{I_{k-1}\times\mathbb{R}^{n}}^{2}}^{2}.
\end{align*}
Recall that $g_{t}=f_{t}^{p}$. Let $p_{0}=1$ and $p_{k}=\chi^{k}=(\frac{n+2}{n})^{k}$
for $k=1,2,\cdots$. Then 
\[
\Vert f_{t}^{p_{k-1}}\Vert_{L_{I_{k}\times\mathbb{R}^{n}}^{2\chi}}^{2}\leq C\left(\vert\alpha\vert^{2}p_{k-1}^{2}+p_{k-1}^{\frac{2}{2-\gamma}}\Vert b\Vert_{L_{t}^{l}L_{x}^{q}}^{\frac{2}{2-\gamma}}\vert\alpha\vert^{\frac{2}{2-\gamma}}+\sigma\frac{4^{k}}{T}\right)\Vert f_{t}^{p_{k-1}}\Vert_{L_{I_{k-1}\times\mathbb{R}^{n}}^{2}}^{2},
\]
or equivalently, 
\[
\Vert f_{t}\Vert_{L_{I_{k}\times\mathbb{R}^{n}}^{2p_{k}}}\leq C^{\frac{1}{2p_{k-1}}}\left(\vert\alpha\vert^{2}p_{k-1}^{2}+p_{k-1}^{\frac{2}{2-\gamma}}\Vert b\Vert_{L_{t}^{l}L_{x}^{q}}^{\frac{2}{2-\gamma}}\vert\alpha\vert^{\frac{2}{2-\gamma}}+\sigma\frac{4^{k}}{T}\right)^{\frac{1}{2p_{k-1}}}\Vert f_{t}\Vert_{L_{I_{k-1}\times\mathbb{R}^{n}}^{2p_{k-1}}}.
\]
Iterate the procedure above to get that
\[
\Vert f_{t}\Vert_{L_{(\frac{T}{2},T)\times\mathbb{R}^{n}}^{\infty}}\leq\left(\prod_{k=1}^{\infty}C^{\frac{1}{2p_{k-1}}}(\vert\alpha\vert^{2}p_{k-1}^{2}+p_{k-1}^{\frac{2}{2-\gamma}}\Vert b\Vert_{L_{t}^{l}L_{x}^{q}}^{\frac{2}{2-\gamma}}\vert\alpha\vert^{\frac{2}{2-\gamma}}+\sigma\frac{4^{k}}{T})^{\frac{1}{2p_{k-1}}}\right)\Vert f_{t}\Vert_{L_{I_{0}\times\mathbb{R}^{n}}^{2}}.
\]
Since $p_{k}=(\frac{n+2}{n})^{k}\leq2^{k}$ and $2-\gamma\leq1$,
we have
\begin{align}
\Vert f_{t}\Vert_{L_{(\frac{T}{2},T)\times\mathbb{R}^{n}}^{\infty}} & \leq\left(\prod_{k=1}^{\infty}C^{\frac{1}{2p_{k-1}}}(\vert\alpha\vert^{2}p_{k-1}^{2}+p_{k-1}^{\frac{2}{2-\gamma}}\Lambda^{\frac{2}{2-\gamma}}\vert\alpha\vert^{\frac{2}{2-\gamma}}+\sigma\frac{4^{k}}{T})^{\frac{1}{2p_{k-1}}}\right)\Vert f_{t}\Vert_{L_{I_{0}\times\mathbb{R}^{n}}^{2}}\nonumber \\
 & \leq\left(\prod_{k=1}^{\infty}C^{\frac{1}{2p_{k-1}}}(\vert\alpha\vert^{2}+\Lambda^{\frac{2}{2-\gamma}}\vert\alpha\vert^{\frac{2}{2-\gamma}}+\frac{\sigma}{T})^{\frac{1}{2p_{k-1}}}(4^{\frac{k}{2-\gamma}})^{\frac{1}{2p_{k-1}}}\right)\Vert f_{t}\Vert_{L_{I_{0}\times\mathbb{R}^{n}}^{2}}\nonumber \\
 & \leq C(\vert\alpha\vert^{2}+\Lambda^{\frac{2}{2-\gamma}}\vert\alpha\vert^{\frac{2}{2-\gamma}}+\frac{\sigma}{T})^{\frac{n+2}{4}}\Vert f_{t}\Vert_{L_{I_{0}\times\mathbb{R}^{n}}^{2}}\nonumber \\
 & =C(\vert\alpha\vert^{2}T+\Lambda^{\frac{2}{2-\gamma}}\vert\alpha\vert^{\frac{2}{2-\gamma}}T+\sigma)^{\frac{n+2}{4}}T^{-\frac{n+2}{4}}\Vert f_{t}\Vert_{L_{I_{0}\times\mathbb{R}^{n}}^{2}}.\label{eq: 3. energy inequality (local use)}
\end{align}
We already proved inequality (\ref{eq: energy estimate -1}), which
implies
\[
\Vert f_{t}\Vert_{L_{I_{0}\times\mathbb{R}^{n}}^{2}}\leq T^{\frac{1}{2}}\exp{\left(C(\vert\alpha\vert^{2}T+\vert\alpha\vert^{\frac{2}{1-\theta}}\Lambda^{\mu}T^{\nu})\right)}\Vert f\Vert_{L_{x}^{2}}.
\]
Inserting this into (\ref{eq: 3. energy inequality (local use)}),
we derive that
\begin{align*}
\Vert f_{t}\Vert_{L_{(\frac{T}{2},T)\times\mathbb{R}^{n}}^{\infty}} & \leq C(\vert\alpha\vert^{2}T+\Lambda^{\frac{2}{2-\gamma}}\vert\alpha\vert^{\frac{2}{2-\gamma}}T+\sigma)^{\frac{n+2}{4}}T^{-\frac{n}{4}}\\
 & \quad\times\exp{\left(C(\vert\alpha\vert^{2}T+\vert\alpha\vert^{\frac{2}{1-\theta}}\Lambda^{\mu}T^{\nu})\right)}\Vert f\Vert_{L_{x}^{2}}.
\end{align*}
Notice that $1-\theta=2-\frac{n}{q}=2-\gamma+\frac{2}{l}$ and recall
that $\mu=\frac{2}{2-\gamma+\frac{2}{l}}$, $\nu=\frac{2-\gamma}{2-\gamma+\frac{2}{l}}$.
Hence 
\[
\vert\alpha\vert^{\frac{2}{2-\gamma}}\Lambda^{\frac{2}{2-\gamma}}T=(\vert\alpha\vert^{\frac{2}{1-\theta}}\Lambda^{\mu}T^{\nu})^{\frac{2-\gamma+\frac{2}{l}}{2-\gamma}},
\]
and $(\vert\alpha\vert^{2}T+\Lambda^{\frac{2}{2-\gamma}}\vert\alpha\vert^{\frac{2}{2-\gamma}}T+\sigma)^{\frac{n+2}{4}}$
can be regarded as a polynomial of parameters $(\vert\alpha\vert^{2}T,\vert\alpha\vert^{\frac{2}{1-\theta}}\Lambda^{\mu}T^{\nu})$,
which can be dominated by 
\[
C\exp{\left(C(\vert\alpha\vert^{2}T+\vert\alpha\vert^{\frac{2}{1-\theta}}\Lambda^{\mu}T^{\nu})\right)}.
\]
So we have 
\[
\Vert f_{T}\Vert_{L_{x}^{\infty}}\leq CT^{-\frac{n}{4}}\exp{\left(C(\vert\alpha\vert^{2}T+\vert\alpha\vert^{\frac{2}{1-\theta}}\Lambda^{\mu}T^{\nu})\right)}\Vert f\Vert_{L_{x}^{2}}.
\]
By duality, i.e. equation (\ref{eq: 3. dual equation})
\[
\Vert f_{T}\Vert_{L_{x}^{2}}\leq CT^{-\frac{n}{4}}\exp{\left(C(\vert\alpha\vert^{2}T+\vert\alpha\vert^{\frac{2}{1-\theta}}\Lambda^{\mu}T^{\nu})\right)}\Vert f\Vert_{L_{x}^{1}}.
\]
Using the Chapman-Kolmogorov equation, one has
\[
\Vert f_{2T}\Vert_{L_{x}^{\infty}}\leq CT^{-\frac{n}{2}}\exp{\left(C(\vert\alpha\vert^{2}T+\vert\alpha\vert^{\frac{2}{1-\theta}}\Lambda^{\mu}T^{\nu})\right)}\Vert f\Vert_{L_{x}^{1}}.
\]
Recall that
\[
f_{2T}(x)=\int_{\rn}f(\xi)\exp(-\psi(x))\Gamma(2T,x;0,\xi)\exp(\psi(\xi))\,d\xi
\]
for any $f\in C_{0}^{\infty}(\rn)$ and that $\psi(x)=\alpha\cdot x$.
Replacing $2T$ by $t$ and dividing both sides by $\exp(-\psi(x))\exp(\psi(\xi))$,
then we have a pointwise upper bound on $\Gamma$ as follows
\[
\Gamma(t,x;0,\xi)\leq\frac{C}{t^{n/2}}\exp\left(C(\vert\alpha\vert^{2}t+\vert\alpha\vert^{\frac{2}{1-\theta}}\Lambda^{\mu}t^{\nu})+\alpha\cdot(x-\xi)\right)
\]
for any $\alpha\in\rn$, where $C$ depends only on $(l,q,n,\lambda)$.
Set $m(t,x)=\min_{\alpha\in\mathbb{R}^{n}}(C(\vert\alpha\vert^{2}t+\vert\alpha\vert^{\frac{2}{1-\theta}}\Lambda^{\mu}t^{\nu})+\alpha\cdot x)$.
Taking the minimum of the right-hand side over all $\alpha\in\rn$,
we can conclude that
\[
\Gamma(t,x;0,\xi)\leq\frac{C}{t^{n/2}}\exp(m(t,x-\xi)).
\]
Finally, we shift $\Gamma(t-\tau,x,0,\xi)$ by $\tau$ to obtain estimate
(\ref{eq: 0 theorem 1 upper bound}). Now the proof is complete.
\end{proof}
We may give an elementary and explicit estimate for the function $m$
appearing in the theorem we just proved, which also gives a more explicit
form of this upper bound. 
\begin{cor}
\label{cor: 3. a explicit upper bound-2} Under the same assumptions
and notations as in Theorem \ref{upperbound}. If $\mu\equiv\frac{2}{2-\gamma+\frac{2}{l}}>1$,
the fundamental solution has upper bound
\begin{equation}
\Gamma(t,x;\tau,\xi)\leq\begin{cases}
\frac{C_{1}}{(t-\tau)^{n/2}}\exp\left(-\frac{1}{C_{2}}\left(\frac{\vert x-\xi\vert^{2}}{t-\tau}\right)\right) & \frac{\vert x\vert^{\mu-2}}{t^{\mu-\nu-1}}<1\\
\frac{C_{1}}{(t-\tau)^{n/2}}\exp\left(-\frac{1}{C_{2}}\left(\frac{\vert x-\xi\vert^{\mu}}{(t-\tau)^{\nu}}\right)^{\frac{1}{\mu-1}}\right) & \frac{\vert x\vert^{\mu-2}}{t^{\mu-\nu-1}}\geq1
\end{cases}\label{eq: 3. upper bound-2}
\end{equation}
where $\Lambda=\Vert b\Vert_{L^{l}(0,T;L^{q}(\rn))}$, $C_{1}=C_{1}(l,q,n,\lambda)$,
$C_{2}=C_{2}(l,q,n,\lambda,\Lambda)$. If $\mu=1$, which implies
$q=\infty$, we can solve for $m(t,x)$ explicitly and obtain 
\begin{equation}
\Gamma(t,x;\tau,\xi)\leq\frac{C_{1}}{(t-\tau)^{n/2}}\exp\left(-\frac{(C_{1}\Lambda(t-\tau)^{\nu}-\vert x-\xi\vert)^{2}}{4C_{1}(t-\tau)}\right).
\end{equation}
\end{cor}

\begin{proof}
Clearly, it is enough to estimate function $m(t,x)$. In this proof,
we denote $C_{1}$ as a constant depending only on $(l,q,n,\lambda)$
and $C_{2}$ a constant depending on $(l,q,n,\lambda,\Lambda)$. Their
values may be different throughout the proof. Notice that $\mu\geq1$.
When $\mu>1$, by taking $\alpha=-\frac{x}{4C_{2}t}$ , we have
\[
m(t,x)\leq\frac{C_{1}\vert x\vert^{2}}{16C_{2}^{2}t}+\frac{C_{1}\Lambda^{\mu}\vert x\vert^{\mu}}{4^{\mu}C_{2}^{\mu}t^{\mu-\nu}}-\frac{\vert x\vert^{2}}{4C_{2}t}=\frac{C_{1}\vert x\vert^{2}}{16C_{2}^{2}t}+\frac{C_{1}\Lambda^{\mu}\vert x\vert^{2}}{4^{\mu}C_{2}^{\mu}t^{1}}\cdot\frac{\vert x\vert^{\mu-2}}{t^{\mu-\nu-1}}-\frac{\vert x\vert^{2}}{4C_{2}t}\leq-\frac{\vert x\vert^{2}}{8C_{2}t}
\]
if $\frac{\vert x\vert^{\mu-2}}{t^{\mu-\nu-1}}<1$. When $\frac{\vert x\vert^{\mu-2}}{t^{\mu-\nu-1}}\geq1$,
we take $\alpha=-\frac{1}{4C_{2}}(\frac{\vert x\vert}{t^{\nu}})^{\frac{1}{\mu-1}}\frac{x}{\vert x\vert}$.
Then one has
\begin{align*}
m(t,x) & \leq\frac{C_{1}\vert x\vert^{\frac{2}{\mu-1}}}{16C_{2}^{2}t^{\frac{2\nu}{\mu-1}-1}}+\frac{C_{1}\Lambda^{\mu}\vert x\vert^{\frac{\mu}{\mu-1}}}{4^{\mu}C_{2}^{\mu}t^{\frac{\nu}{\mu-1}}}-\frac{\vert x\vert^{\frac{\mu}{\mu-1}}}{C_{2}t^{\frac{\nu}{\mu-1}}}\\
 & =\frac{C_{1}\vert x\vert^{\frac{\mu}{\mu-1}}}{16C_{2}^{2}t^{\frac{\nu}{\mu-1}}}\cdot\frac{\vert x\vert^{\frac{2-\mu}{\mu-1}}}{t^{\frac{v-\mu+1}{\mu-1}}}+\frac{C_{1}\Lambda^{\mu}\vert x\vert^{\frac{\mu}{\mu-1}}}{4^{\mu}C_{2}^{\mu}t^{\frac{\nu}{\mu-1}}}-\frac{\vert x\vert^{\frac{\mu}{\mu-1}}}{C_{2}t^{\frac{\nu}{\mu-1}}}\leq-\frac{\vert x\vert^{\frac{\mu}{\mu-1}}}{2C_{2}t^{\frac{\nu}{\mu-1}}}.
\end{align*}
Now consider the case that $\mu=1$. To obtain $m(t,x)=\min_{\alpha\in\mathbb{R}^{n}}(C_{1}(\vert\alpha\vert^{2}t+\vert\alpha\vert\Lambda t^{\nu})+\alpha\cdot x)$,
it is easy to see that $\alpha$ must be in opposite direction of
$x$, i.e. $\frac{\alpha}{\vert\alpha\vert}=-\frac{x}{\vert x\vert}$.
So we only need to find the minimum of the polynomial $C_{1}t\vert\alpha\vert^{2}+(C_{1}\Lambda t^{\nu}-\vert x\vert)\vert\alpha\vert$,
which is obtained at $\vert\alpha\vert=-\frac{C_{1}\Lambda t^{\nu}-\vert x\vert}{2C_{1}t}$
and the value is
\[
m(t,x)=-\frac{(C_{1}\Lambda t^{\nu}-\vert x\vert)^{2}}{4C_{1}t}.
\]
Now the proof is complete. 
\end{proof}
Recall that in dimension three, any Leray-Hopf weak solution to the
Navier-Stokes equations satisfies
\[
u\in L^{\infty}(0,T;L^{2}(\rt))\cap L^{2}(0,T;H^{1}(\rt)).
\]
Clearly $L^{2}(0,T;H^{1}(\rt))\subset L^{2}(0,T;L^{6}(\rt))$, thus
$\gamma=\frac{3}{2}$ for both function spaces. Notice that by interpolation,
$u\in L^{l}(0,T;L^{q}(\rt))$ for any $l\in[2,\infty]$ and $q\in[2,6]$
satisfying $\frac{2}{l}+\frac{3}{q}=\frac{3}{2}$. This is an interesting
case for which we have the following theorem.
\begin{thm}
\label{The10}Suppose $n=3$, and conditions (\ref{eq: 1. ellipticity})
and (\ref{eq: 1. divergence free}) hold for $b\in L^{l}(0,T;L^{q}(\rt))$
satisfying $\frac{2}{l}+\frac{3}{q}=\frac{3}{2}$. Then the fundamental
solution $\Gamma$ to (\ref{eq: 1. problem equation}) has the upper
bound 

\begin{equation}
\Gamma(t,x;\tau,\xi)\leq\begin{cases}
\frac{C_{1}}{(t-\tau)^{3/2}}\exp\left(-\frac{1}{C_{2}}\left(\frac{\vert x-\xi\vert^{2}}{t-\tau}\right)\right) & \frac{\vert x\vert^{l-4}}{t^{l-2}}<1\\
\frac{C_{1}}{(t-\tau)^{3/2}}\exp\left(-\frac{1}{C_{2}}\left(\frac{\vert x-\xi\vert^{4}}{t-\tau}\right)^{\frac{1}{3-\frac{4}{l}}}\right) & \frac{\vert x\vert^{l-4}}{t^{l-2}}\geq1,
\end{cases}
\end{equation}
where $\Lambda=\Vert b\Vert_{L^{l}(0,T;L^{q}(\rt))}$, $C_{1}=C_{1}(l,q,n,\lambda)$,
$C_{2}=C_{2}(l,q,\lambda,\Lambda)$. Here we set $\frac{\vert x\vert^{l-4}}{t^{l-2}}=\frac{\vert x\vert}{t}$
when $l=\infty$.
\end{thm}

Another example is $L^{\infty}(0,T;L^{n}(\rn))$, which is not covered
in the classical paper \cite{aronson1968non} by Aronson and \cite{ladyzhenskaia1988linear}.
Here with the divergence-free condition on $b$, we have obtained
a Gaussian upper bound, which yields the first half of Theorem \ref{The-fundamental-solution-ln}.
\begin{thm}
\label{thm: 3. ln upper bound}The fundamental solution to (\ref{eq: 1. problem equation})
satisfying conditions (\ref{eq: 1. ellipticity}) and (\ref{eq: 1. divergence free}),
with $b\in L^{\infty}(0,T;L^{n}(\rn))$ has the Gaussian upper bound
\[
\Gamma(t,x;\tau,\xi)\leq\frac{C_{1}}{(t-\tau)^{n/2}}\exp\left(-\frac{1}{C_{2}}\left(\frac{\vert x-\xi\vert^{2}}{t-\tau}\right)\right),
\]
where $\Lambda=\Vert b\Vert_{L^{\infty}(0,T;L^{n}(\rn))}$, $C_{1}=C_{1}(n,\lambda)$,
$C_{2}=C_{2}(n,\lambda,\Lambda)$.
\end{thm}

Since we know that $\int_{\rn}\Gamma(t,x;\tau,\xi)d\xi=1$ and we
have proved the upper bound in Corollary \ref{cor: 3. a explicit upper bound-2},
which is of exponential decay in space, we can derive a lower bound
for $\Gamma$ in the following form. This proposition will be used
in the proof of a pointwise lower bound later.
\begin{prop}
\label{prop: 3. lower bound in a cone}Take the fundamental solution
of (\ref{eq: 1. problem equation}) satisfying conditions (\ref{eq: 1. ellipticity})
and (\ref{eq: 1. divergence free}), for any $\delta\in(0,1)$ and
$t-\tau$ small enough, we have 
\[
\int_{B(x,\tilde{R}(t-\tau))}\Gamma(t,x;\tau,\xi)d\xi\geq\delta,
\]
where $\tilde{R}(\cdot)$ is a function defined as follows
\[
\tilde{R}(t)=\begin{cases}
Ct^{1/2} & \text{if }\gamma=1\\
Ct^{(2-\gamma)/2}\ln\frac{1}{t} & \text{if }\gamma>1,
\end{cases}
\]
 $B(x,r)$ is the ball of radius $r$ and center $x$, and $C$ depends
only on $(\delta,l,q,n,\lambda,\Lambda)$. 
\end{prop}

\begin{proof}
Firstly, when $\mu>1$, we have
\begin{equation}
\Gamma(t,x;\tau,\xi)\leq h_{1}(t-\tau,x-\xi)+h_{2}(t-\tau,x-\xi),\label{eq: 3. rough explicit upper bound}
\end{equation}
where
\[
h_{1}(t,x)=\frac{C_{1}}{t^{n/2}}\exp\left(-\frac{1}{C_{2}}\left(\frac{\vert x\vert^{2}}{t}\right)\right),\quad h_{2}(t,x)=\frac{C_{1}}{t^{n/2}}\exp\left(-\frac{1}{C_{2}}\left(\frac{\vert x\vert^{\mu}}{t^{\nu}}\right)^{\frac{1}{\mu-1}}\right).
\]
Thus it is enough to prove that 
\[
\int_{B(x,\tilde{R}(t-\tau))^{c}}h_{1}(t-\tau,x-\xi)+h_{2}(t-\tau,x-\xi)d\xi\leq1-\delta.
\]
Without lose of generality, we can assume $\tau=0$ and $x=0$. By
the following change of variable
\[
\int_{B(0,RC_{2}^{1/2}t^{1/2})^{c}}\frac{C_{1}}{t^{n/2}}\exp\left(-\frac{1}{C_{2}}\left(\frac{\vert\xi\vert^{2}}{t}\right)\right)d\xi=C\int_{B(0,R)^{c}}\exp\left(-\vert\xi\vert^{2}\right)d\xi,
\]
we have
\[
\int_{B(x,R_{1}(t))^{c}}\frac{C_{1}}{t^{n/2}}\exp\left(-\frac{1}{C_{2}}\left(\frac{\vert\xi\vert^{2}}{t}\right)\right)d\xi\leq\frac{1-\delta}{2}
\]
with $R_{1}(t)=Ct^{1/2}$ for some sufficiently large constant $C>0$.
For the second term, since $\frac{\nu}{\mu}=\frac{2-\gamma}{2}\leq\frac{1}{2}$,
it follows that 
\begin{multline*}
\int_{B(0,RC_{2}^{(\mu-1)\nu/\mu}t^{\nu/\mu})^{c}}\frac{C_{1}}{t^{n/2}}\exp\left(-\frac{1}{C_{2}}\left(\frac{\vert\xi\vert^{\mu}}{t^{\nu}}\right)^{\frac{1}{\mu-1}}\right)d\xi=\\
\frac{C}{t^{n(\frac{1}{2}-\frac{\nu}{\mu})}}\int_{B(0,R)^{c}}\exp\left(-\vert\xi\vert^{\frac{\mu}{\mu-1}}\right)d\xi.
\end{multline*}
Setting $\Phi(R)=\int_{B(0,R)^{c}}\exp\left(-\vert\xi\vert^{\frac{\mu}{\mu-1}}\right)d\xi$,
then one has $\Phi(R)\leq Ce^{-R}$ by $\frac{\mu}{\mu-1}>1$ . So
setting $R=C(1-(\frac{1}{2}-\frac{\nu}{\mu})\ln t)$ for some $C$,
we obtain that
\[
\int_{B(0,R)^{c}}\exp\left(-\vert\xi\vert^{\frac{\mu}{\mu-1}}\right)d\xi\leq Ce^{-R}\leq\frac{t^{n(\frac{1}{2}-\frac{\nu}{\mu})}}{C_{1}}\cdot\frac{1-\delta}{2},
\]
and therefore
\[
\int_{B(0,R_{2}(t))^{c}}\frac{C_{1}}{t^{n/2}}\exp\left(-\frac{1}{C_{2}}\left(\frac{\vert\xi\vert^{\mu}}{t^{\nu}}\right)^{\frac{1}{\mu-1}}\right)d\xi\leq\frac{1-\delta}{2}
\]
with 
\[
R_{2}(t)=Ct^{\nu/\mu}(1+(\frac{1}{2}-\frac{\nu}{\mu})\ln\frac{1}{t})=Ct^{(2-\gamma)/2}(1+(\frac{1}{2}-\frac{2-\gamma}{2})\ln\frac{1}{t})
\]
for some constant $C$. When $t$ is small enough, $R_{1}(t)\leq R_{2}(t)$
and we obtain the radius $\tilde{R}(t)=R_{2}(t)$. 

When $\mu=1$, we have $\nu=\frac{2-\gamma}{2}\leq\frac{1}{2}$ and,
by using the elementary inequality $(a-b)^{2}+b^{2}\geq\frac{a^{2}}{2}$,

\begin{align*}
 & \int_{B(0,Rt^{\nu})^{c}}\frac{C_{1}}{t^{n/2}}\exp\left(-\frac{(C_{1}\Lambda t^{\nu}-\vert\xi\vert)^{2}}{4C_{1}t}\right)d\xi\\
 & \quad=\frac{C}{t^{n(\frac{1}{2}-\nu)}}\int_{B(0,R)^{c}}\exp\left(-\frac{(C_{1}\Lambda-\vert\xi\vert)^{2}}{C_{1}t^{1-2\nu}}\right)d\xi\\
 & \quad\leq\frac{C}{t^{n(\frac{1}{2}-\nu)}}\int_{B(0,R)^{c}}\exp\left(-\frac{(C_{1}\Lambda-\vert\xi\vert)^{2}}{C_{1}}\right)d\xi\\
 & \quad\leq\frac{C}{t^{n(\frac{1}{2}-\nu)}}\int_{B(0,R)^{c}}\exp\left(-\frac{\vert\xi\vert^{2}}{2C_{1}}+C_{1}^{2}\Lambda^{2}\right)d\xi.
\end{align*}
Let $\Phi(R)=\int_{B(0,R)^{c}}\exp\left(-\vert\xi\vert^{2}\right)d\xi$.
Then $\Phi(R)\leq Ce^{-R}$ for some universal constant. So we still
take 
\[
\tilde{R}(t)=Ct^{(2-\gamma)/2}(1+(\frac{1}{2}-\frac{2-\gamma}{2})\ln\frac{1}{t})
\]
to obtain 
\[
\int_{B(0,\tilde{R}(t))^{c}}\frac{C_{1}}{t^{n/2}}\exp\left(-\frac{(C_{1}\Lambda t^{\nu}-\vert\xi\vert)^{2}}{4C_{1}t}\right)d\xi\leq1-\delta.
\]
Clearly, when $\gamma=1$, $\tilde{R}(t)$ is just $Ct^{1/2}$. Since
we only need $\tilde{R}(t)$ for small $t$, and under this condition
$\ln\frac{1}{t}\gg1$. Thus taking $\tilde{R}(t)=Ct^{(2-\gamma)/2}\ln\frac{1}{t}$
when $\gamma>1$ will do.
\end{proof}
\begin{rem*}
\textcolor{black}{Although estimate (\ref{eq: 3. rough explicit upper bound})
seems better than (\ref{eq: 3. upper bound-2}), actually it can be
shown that this observation will not affect the result. So based on
Corollary \ref{cor: 3. a explicit upper bound-2}, this $\tilde{R}(t)$
is the smallest cone radius such that we can derive a lower bound
of this form inside the cone. }
\end{rem*}

\subsection{The lower bound}

Let $T>0$ and $x\in\rn$. To prove the lower bound, Nash's idea is
to consider the quantity
\[
G_{r}(t,x)=\int_{\rn}\ln(\Gamma(T,x;T-t,\xi))\mu_{r}(d\xi)=\int_{\rn}\ln(\Gamma_{T}^{*}(t,\xi;0,x))\mu_{r}(d\xi)
\]
for $t\in[0,T]$, where $\mu_{r}(x)=\frac{1}{r^{n/2}}\exp\left(-\frac{\pi\vert x\vert^{2}}{r}\right)$
as defined in Lemma \ref{lem: poincare inequality}. Then Jensen's
inequality implies that $G_{r}(t,x)\leq0$. We will write it as $G(t,x)$
if $r=1$. If we have $G_{r}(T,x)>-C$ for some positive constant
$C$, then we can derive a lower bound for $\Gamma(T,x;0,\xi)$. Consider
the time derivative of $G_{r}(t,x)$
\begin{align}
G_{r}'(t,x) & =\int_{\mathbb{R}^{n}}\partial_{t}\ln(\varGamma(T,x;T-t,\xi))\;\mu_{r}(d\xi)\nonumber \\
 & =\int_{\mathbb{R}^{n}}\left\langle \frac{2\pi\xi}{r},a(T-t,\xi)\cdot\nabla_{\xi}\ln\varGamma(T,x;T-t,\xi)\right\rangle \mu_{r}(d\xi)\nonumber \\
 & \quad+\int_{\mathbb{R}^{n}}\left\langle \nabla_{\xi}\ln\varGamma(T,x;T-t,\xi),a(T-t,\xi)\cdot\nabla_{\xi}\ln\varGamma(T,x;T-t,\xi)\right\rangle \mu_{r}(d\xi)\nonumber \\
 & \quad+\int_{\mathbb{R}^{n}}\left\langle b(T-t,\xi),\nabla_{\xi}\ln\varGamma(T,x;T-t,\xi)\right\rangle \mu_{r}(d\xi).\label{eq: 4. equation for G 1}
\end{align}
In the following subsections, we will estimate $G_{r}(t,x)$ under
varies conditions on $b$ and hence obtain a lower bound of $\Gamma$. 

We will separate the critical and supercritical cases. In critical
case $\gamma=1$, we will only consider the case where $l=\infty,q=n$,
which is the only case that regularity theory is missing. Since in
the critical case, $\Vert b\Vert_{L_{t}^{\infty}L_{x}^{n}}$ is invariant
under scaling, we do not need to worry about explicitly how the constant
depends on $\Lambda$. Hence we will only need to estimate $G(1,x)$
and obtain the estimate of $G(t,x)$ for all $t$ by scaling. In supercritical
case, in order to use scaling, we need to find out how the constants
appearing in lower bounds depend on $\Lambda$, and therefore it is
not a good idea to use the scaling argument. So we will alter the
strategy to estimate $G_{r}(t,x)$ for all $t$ directly. 

\subsubsection{Critical case $L^{\infty}(0,T;L^{n}(\protect\rn))$}

In the critical case $L^{\infty}(0,T;L^{n}(\rn))$, we can obtain
the following Gaussian lower bound and hence complete the proof of
Theorem \ref{The-fundamental-solution-ln}. 
\begin{lem}
\label{lem: 4. Lower bound}For any $x\in B(0,2)$, there is a constant
$C>0$ depending only on $n,\lambda,\Lambda=\Vert b\Vert_{L^{\infty}(0,T;L^{n}(\rn))}$,
such that

\begin{equation}
G(1,x)=\int_{\mathbb{R}^{n}}\ln(\varGamma(1,x;0,\xi))\;\mu(d\xi)\geq-C,\label{eq: 4.  lowerbound step Ln-infty 1}
\end{equation}
and hence 
\begin{equation}
\varGamma(2,x;0,\xi)\geq e^{-2C},\qquad x,\xi\in B(0,2).\label{eq: 4.  lowerbound step Ln-infty 2}
\end{equation}
\end{lem}

\begin{proof}
If we fix $x\in B(0,2)$, $T=1$ and $r=1$ in equation (\ref{eq: 4. equation for G 1}),
then condition (\ref{eq: 1. ellipticity}) and $b\in L^{\infty}(0,T;L^{n}(\rn))$
implies
\begin{align*}
G'(t,x) & \geq-\frac{C}{\lambda}\Vert\xi\Vert_{L^{2}(\mu)}\Vert\nabla_{\xi}\ln\varGamma(1,x;1-t,\xi)\Vert_{L^{2}(\mu)}+\lambda\Vert\nabla_{\xi}\ln\varGamma(1,x;1-t,\xi)\Vert_{L^{2}(\mu)}^{2}\\
 & \quad-C\Vert b(1-t)\Vert_{L^{n}(\mu)}\Vert\nabla_{\xi}\ln\varGamma(1,x;1-t,\xi)\Vert_{L^{2}(\mu)}\\
 & \geq-C(n,\lambda,\Lambda)+\frac{\lambda}{2}\Vert\nabla_{\xi}\ln\varGamma(1,x;1-t,\xi)\Vert_{L^{2}(\mu)}^{2}\\
 & \geq-C(n,\lambda,\Lambda)+C(\lambda)\left\Vert \ln\varGamma(1,x;1-t,\cdot)-G(t,x)\right\Vert _{L^{2}\left(\mu\right)}^{2},
\end{align*}
where the last step follows the Poincaré-Wirtinger inequality (\ref{eq: 4. poincare inequality 1})
for the Gaussian measure\textcolor{black}{. The rest of the argument
of the proof follows exactly the same as in \cite{stroock1988diffusion}.
For completeness, we include the full proof here. S}ince $G(t,x)\leq0$,
for any $K>0$, using $(a-b)^{2}\geq\frac{a^{2}}{2}-b^{2}$ we have
\begin{align*}
 & \left\Vert \ln\varGamma(1,x;1-t,\cdot)-G(t,x)\right\Vert _{L^{2}\left(\mu\right)}^{2}\\
 & \quad\geq\int_{\left\{ \ln\varGamma(1,x;1-t,\xi)\geq-K\right\} }\left(\ln\varGamma(1,x;1-t,\xi)-G(t,x)\right)^{2}\mu(d\xi)\\
 & \quad=\int_{\left\{ \ln\varGamma(1,x;1-t,\xi)\geq-K\right\} }\left(\ln\varGamma(1,x;1-t,\xi)-G(t,x)-K+K\right)^{2}\mu(d\xi)\\
 & \quad\geq\frac{1}{2}\int_{\left\{ \ln\varGamma(1,x;1-t,\xi)\geq-K\right\} }\left(\ln\varGamma(1,x;1-t,\xi)+K-G(t,x)\right)^{2}\mu(d\xi)-K^{2}\\
 & \quad\geq\frac{1}{2}\int_{\left\{ \ln\varGamma(1,x;1-t,\xi)\geq-K\right\} }G(t,x)^{2}\mu(d\xi)-K^{2}\\
 & \quad=\frac{1}{2}G(t,x)^{2}\mu\{\ln\varGamma(1,x;1-t,\xi)\geq-K\}-K^{2}.
\end{align*}
By the upper bound in Theorem \ref{thm: 3. ln upper bound}, we can
find a $R>0$ depending only on $n,\lambda,\Lambda$ such that 
\[
\int_{\left\{ |\xi|>R\right\} }\varGamma(1,x;1-t,\xi)\;d\xi<\frac{1}{4},\qquad\mbox{ for all }t\in(0,1],x\in B(0,2).
\]
Also the upper bound implies that there exists a $M>0$ depending
only on $n,\lambda$ such that $\varGamma(1,x;1-t,\xi)\leq M$ for
$t\in[\frac{1}{2},1]$. So one has 
\begin{align*}
\frac{3}{4} & \leq\int_{\left\{ |\xi|\leq R\right\} }\varGamma(1,x;1-t,\xi)\;d\xi\\
 & \leq\vert B(0,R)\vert e^{-K}+M(2\pi)^{n/2}e^{R^{2}/2}\mu\left\{ \xi\in\mathbb{R}^{n}:\ln\varGamma(x,1;\xi,1-t)\geq-K\right\} .
\end{align*}
Now we can choose a large enough $K$ depending on $R$ such that
\[
\mu\left\{ \xi\in\mathbb{R}^{n}:\ln\varGamma(x,1;\xi,1-t)\geq-K\right\} \geq\frac{1}{2M(2\pi)^{n/2}e^{R^{2}/2}}.
\]
Combining these all together, we have 
\begin{equation}
\left\Vert \ln\varGamma(1,x;1-t,\cdot)-G(t,x)\right\Vert _{L^{2}\left(\mu\right)}^{2}\geq-K^{2}+\frac{1}{4M(2\pi)^{n/2}e^{R^{2}/2}}G(t,x)^{2},\label{eq: 4. lower bound l2-infty estimate 1}
\end{equation}
and therefore
\[
G'(t,x)\geq-C_{1}+C_{2}G(t,x)^{2}.
\]
Together with the fact that $G\leq0$, it follows from Lemma \ref{lem: 4. differential inequality}
that 
\begin{equation}
G(1,x)\geq\min\left\{ -C_{1}-2\sqrt{\frac{C_{1}}{C_{2}}},-\frac{8}{3C_{2}}\right\} =-C.\label{eq:str0.1}
\end{equation}
Finally the Chapman-Kolmogorov equation and Jensen's inequality yield
that
\begin{align*}
\ln\varGamma(2,x;0,\xi) & =\ln\left(\int_{\mathbb{R}^{n}}\varGamma\left(2,x;1,z\right)\varGamma(1,z;0,\xi)dz\right)\\
 & =\ln\left(\int_{\mathbb{R}^{n}}(2\pi)^{n/2}e^{|z|^{2}/2}\varGamma\left(2,x;1,z\right)\varGamma(1,z;0,\xi)\mu(dz)\right)\\
 & \geq\ln\left(\int_{\mathbb{R}^{n}}\varGamma\left(2,x;1,z\right)\varGamma(1,z;0,\xi)\mu(dz)\right)\\
 & \geq\int_{\mathbb{R}^{n}}\ln\varGamma\left(2,x;1,z\right)\mu(dz)+\int_{\mathbb{R}^{n}}\ln\varGamma(1,z;0,\xi)\mu(dz)\\
 & \geq-2C,
\end{align*}
where the estimate of $\int_{\mathbb{R}^{n}}\ln\varGamma\left(2,x;1,z\right)\mu(dz)\geq-C$
is simply a shift in time and $\int_{\mathbb{R}^{n}}\ln\varGamma(1,z;0,\xi)\mu(dz)\geq-C$
can be obtained by analyzing the dual operator with the same argument.
\end{proof}
\begin{proof}
[Proof of Theorem \ref{The-fundamental-solution-ln}]Recall the scaling
invariant property, i.e. for any $\rho>0$ and $z\in\mathbb{R}^{n}$,
\begin{equation}
\varGamma(\rho^{2}t,\rho x+z;0,\rho\xi+z)=\rho^{-n}\varGamma^{(a_{\rho,z},b_{\rho,z})}(t,x;0,\xi)
\end{equation}
where $a_{\rho,z}(t,x)=a(\rho^{2}t,\rho x+z)$, $b_{\rho,z}(t,x)=\rho b(\rho^{2}t,\rho x+z)$
and $\varGamma^{(a,b)}$ is the fundamental solution associated with
$(a,b)$. The transformation $(a,b)\rightarrow(a_{\rho,z},b_{\rho,z})$
preserves the ellipticity constant $\lambda$ of $a$ and more importantly
the $L^{\infty}(0,T;L^{n}(\rn))$ norms of $b$. So we may apply Lemma
\ref{lem: 4. Lower bound} to $\varGamma^{(a_{\rho,z},b_{\rho,z})}$
to deduce that 
\begin{equation}
\varGamma(2t,x;0,\xi)\geq\frac{e^{-2C}}{t^{n/2}},\quad\vert\xi-x\vert<4t^{\frac{1}{2}}.
\end{equation}
Next, to obtain a lower bound for all $x,\xi\in\rn$, we use the Chapman-Kolmogorov
equation. Suppose $\vert x-\xi\vert^{2}\in[k,k+1)$, we set $\xi_{m}=\xi+\frac{m}{k+1}(x-\xi)$
and $B_{m}=B(\xi_{m},\frac{1}{k^{1/2}})$ for $0\leq m\leq k+1$.
For any $z_{m}\in B_{m}$, we will have $\vert z_{m}-z_{m-1}\vert<\frac{3}{k^{1/2}}$
and hence $\Gamma(\frac{2m}{k+1},z_{m};\frac{2(m-1)}{k+1},z_{m-1})\geq k^{n/2}e^{-2C}$.
By the Chapman-Kolmogorov equation we obtain that
\begin{align*}
\Gamma(2,x;0,\xi) & \geq\int_{B_{1}}\cdots\int_{B_{k}}\prod_{m=1}^{k+1}\Gamma(\frac{2m}{k+1},z_{m};\frac{2(m-1)}{k+1},z_{m-1})dz_{1}\cdots dz_{k}\\
 & \geq(k^{n/2}e^{-2C})^{k+1}(\vert B(1)\vert k^{-n/2})^{k}.
\end{align*}
Choose a constant $C'$ such that $e^{-C'}\leq\vert B(1)\vert e^{-2C}$
to obtain 
\[
\Gamma(2,x;0,\xi)\geq e^{-2C}e^{-C'\vert x-\xi\vert^{2}}.
\]
Finally, we use scaling again to obtain that
\[
\Gamma(t,x;0,\xi)\geq\frac{1}{Ct^{n/2}}\exp\left(-C\frac{\vert x-\xi\vert^{2}}{t}\right).
\]
Now the proof is complete.
\end{proof}

\subsubsection{Supercritical cases}

Now we consider the supercritical case $1<\gamma<2$. We will estimate
$G_{r}(t,x)$ directly without using scaling argument.
\begin{lem}
\label{lem: 4. Lower bound general}Suppose $q\geq2$, $l\geq2$,
$1<\gamma<2$ and $\tilde{R}(t)$ is defined as in Proposition \ref{prop: 3. lower bound in a cone}.
For any $\kappa>0$, $x\in B(0,\kappa\tilde{R}(t))$ and $t>0$ small
enough, there is a constant $C>0$ depending only on $\kappa,l,q,n,\lambda,\Lambda=\Vert b\Vert_{L^{l}(0,T;L^{q}(\rn))}$,
such that

\begin{equation}
G_{r}(t,x)\geq-C(\lambda)(\frac{t}{r}+r^{-n/q}t^{\frac{l-2}{l}}\Lambda^{2})-C(\frac{r}{t})^{n/2+1}\exp\left(\frac{\pi\vert\tilde{R}(t)\vert^{2}}{Cr}\right).\label{eq: 4.  lowerbound general step 1}
\end{equation}
\end{lem}

\begin{proof}
We fix a $T>0$. By the definition of $G_{r}(t,x)$, for $0<t_{1}<t_{2}\leq T$,
we can deduce that

\begin{align*}
 & G_{r}(t_{2},x)-G_{r}(t_{1},x)\\
 & \quad=\int_{t_{1}}^{t_{2}}\int_{\rn}\frac{\partial_{s}\Gamma(T,x;T-s,\xi)}{\Gamma(T,x;T-s,\xi)}\mu_{r}(d\xi)ds\\
 & \quad=\int_{t_{1}}^{t_{2}}\int_{\rn}\langle\frac{2\pi\xi}{r},a(T-s,\xi)\cdot\nabla_{\xi}\ln\Gamma(T,x;T-s,\xi)\rangle\mu_{r}(d\xi)ds\\
 & \quad\quad+\int_{t_{1}}^{t_{2}}\int_{\rn}\langle\nabla_{\xi}\ln\Gamma(T,x;T-s,\xi),a(T-s,\xi)\cdot\nabla_{\xi}\ln\Gamma(T,x;T-s,\xi)\rangle\mu_{r}(d\xi)ds\\
 & \quad\quad+\int_{t_{1}}^{t_{2}}\int_{\rn}\langle b(T-s,\xi),\nabla_{\xi}\ln\Gamma(T,x;T-s,\xi)\rangle\mu_{r}(d\xi)ds\\
 & \quad\geq-\int_{t_{1}}^{t_{2}}\frac{2\pi}{\lambda r}\Vert\xi\Vert_{L^{2}(\mu_{r})}\Vert\nabla_{\xi}\ln\Gamma(T,x;T-s,\cdot)\Vert_{L^{2}(\mu_{r})}ds\\
 & \quad\quad+\int_{t_{1}}^{t_{2}}\lambda\Vert\nabla_{\xi}\ln\Gamma(T,x;T-s,\cdot)\Vert_{L^{2}(\mu_{r})}^{2}ds\\
 & \quad\quad-\int_{t_{1}}^{t_{2}}\Vert b(T-s,\cdot)\Vert_{L^{q}(\rn)}\Vert\mu_{r}^{\frac{1}{2}}\Vert_{L^{\frac{2q}{q-2}}(\rn)}\Vert\nabla_{\xi}\ln\Gamma(T,x;T-s,\cdot)\Vert_{L^{2}(\mu_{r})}ds\\
 & \quad\geq-\int_{t_{1}}^{t_{2}}\frac{C(\lambda)}{r^{2}}\Vert\xi\Vert_{L^{2}(\mu_{r})}^{2}+C(\lambda)\Vert b(T-s,\cdot)\Vert_{L^{q}(\rn)}^{2}\Vert\mu_{r}^{\frac{1}{2}}\Vert_{L^{\frac{2q}{q-2}}(\rn)}^{2}ds\\
 & \quad\quad+\frac{\lambda}{2}\int_{t_{1}}^{t_{2}}\Vert\nabla_{\xi}\ln\Gamma(T,x;T-s,\cdot)\Vert_{L^{2}(\mu_{r})}^{2}ds.
\end{align*}
Here we set $\frac{2q}{q-2}=\infty$ when $q=2$. Since $l\geq2$,
we have 
\[
\int_{0}^{T}\Vert b(T-s,\cdot)\Vert_{L^{q}(\rn)}^{2}ds<\infty.
\]
For the last term in the equation above, we use the Poincaré-Wirtinger
inequality and obtain that 
\begin{multline*}
\int_{t_{1}}^{t_{2}}\Vert\nabla_{\xi}\ln\Gamma(T,x;T-s,\cdot)\Vert_{L^{2}(\mu_{r})}^{2}ds\geq\\
Cr^{-1}\int_{t_{1}}^{t_{2}}\int_{\rn}\vert\ln\Gamma(T,x;T-s,\xi)-G_{r}(s,x)\vert^{2}\mu_{r}(d\xi)ds.
\end{multline*}
 Since $G_{r}(t,x)\leq0$, using $(a-b)^{2}\geq\frac{a^{2}}{2}-b^{2}$,
the right-hand side can be estimated as
\begin{align*}
 & \int_{t_{1}}^{t_{2}}\int_{\rn}\vert\ln\Gamma(T,x;T-s,\xi)-G_{r}(s,x)\vert^{2}\mu_{r}(d\xi)ds\\
 & \quad\geq\int_{t_{1}}^{t_{2}}\int_{\left\{ \ln\varGamma(T,x;T-s,\xi)\geq-1\right\} }\left(\ln\varGamma(T,x;T-s,\xi)-G_{r}(s,x)-1+1\right)^{2}\mu_{r}(d\xi)ds\\
 & \quad\geq\frac{1}{2}\int_{t_{1}}^{t_{2}}\int_{\left\{ \ln\varGamma(T,x;T-s,\xi)\geq-1\right\} }\left(\ln\varGamma(T,x;T-s,\xi)+1-G_{r}(s,x)\right)^{2}\mu_{r}(d\xi)-1ds\\
 & \quad\geq\frac{1}{2}\int_{t_{1}}^{t_{2}}\int_{\left\{ \ln\varGamma(T,x;T-s,\xi)\geq-1\right\} }G_{r}(s,x)^{2}\mu_{r}(d\xi)-1ds\\
 & \quad=\frac{1}{2}\int_{t_{1}}^{t_{2}}G_{r}(s,x)^{2}\mu_{r}\{\ln\varGamma(T,x;T-s,\xi)\geq-1\}-1ds.
\end{align*}
By Proposition \ref{prop: 3. lower bound in a cone}, for any $x\in B(0,\kappa\tilde{R}(t))$,
we have 
\[
\int_{B(0,(C+\kappa)\tilde{R}(t))}\Gamma(T,x;T-t,\xi)d\xi\geq\frac{1}{2},
\]
which implies that 
\[
\mu_{r}\{\ln\varGamma(T,x;T-t,\xi)\geq-1\}\frac{Cr^{n/2}}{T^{n/2}}\exp\left(\frac{\pi\vert\tilde{R}(T)\vert^{2}}{Cr}\right)+\vert B(0,(C+\kappa)\tilde{R}(t))\vert e^{-1}\geq\frac{1}{2}
\]
for $t\in[\frac{T}{2},T]$. So if we take $T>0$ small enough such
that $\vert B(0,(C+\kappa)R(t))\vert e^{-1}\leq\frac{1}{4}$ for all
$t\in[0,T]$, then
\[
\mu_{r}\{\ln\varGamma(T,x;T-t,\xi)\geq-1\}\geq C(\frac{T}{r})^{n/2}\exp\left(-\frac{\pi\vert\tilde{R}(T)\vert^{2}}{Cr}\right).
\]
Also it is easy to calculate that $\Vert\xi\Vert_{L^{2}(\mu_{r})}^{2}=r$
and
\[
\Vert\mu_{r}^{\frac{1}{2}}\Vert_{L^{\frac{2q}{q-2}}(\rn)}^{2}=\Vert\mu_{r}\Vert_{L^{\frac{q}{q-2}}(\rn)}=C(q)r^{-n/q}.
\]
We now can conclude that

\begin{align*}
G_{r}(t_{2},x)-G(t_{1},x) & \geq-\int_{t_{1}}^{t_{2}}\frac{C(\lambda)}{r}+C(\lambda)r^{-n/q}\Vert b(T-s,\cdot)\Vert_{L^{q}(\rn)}^{2}+Cr^{-1}ds\\
 & \quad+Cr^{-1}(\frac{T}{r})^{n/2}\exp\left(-\frac{\pi\vert\tilde{R}(T)\vert^{2}}{Cr}\right)\int_{t_{1}}^{t_{2}}G_{r}(s,x)^{2}ds,
\end{align*}
for $\frac{T}{2}\leq t_{1}<t_{2}\leq T$. By Lemma \ref{lem: 2. generalized differential inequality}
and $l\geq2$, we have 
\[
G_{r}(T,x)\geq-C(\lambda)(\frac{T}{r}+r^{-n/q}T^{\frac{l-2}{l}}\Lambda^{2})-C(\frac{r}{T})^{n/2+1}\exp\left(\frac{\pi\vert\tilde{R}(T)\vert^{2}}{Cr}\right).
\]
\end{proof}
\begin{proof}
[Proof of Theorem \ref{lowerbound}]For $x,\xi\in B(0,\kappa\tilde{R}(t))$,
by using the Chapman-Kolmogorov equation we obtain that

\begin{align*}
\ln\varGamma(2T,x;0,\xi) & =\ln\int_{\rn}\varGamma(2T,x;T,z)\varGamma(T,z;0,\xi)dz\\
 & \geq\ln\int_{\rn}r^{n/2}\varGamma(2T,x;T,z)\varGamma(T,z;0,\xi)\mu_{r}(dz)\\
 & \geq\frac{n}{2}\ln r+\int_{\rn}\ln\varGamma(2T,x;T,z)\varGamma(T,z;0,\xi)\mu_{r}(dz)\\
 & =\frac{n}{2}\ln r+\int_{\rn}\ln\varGamma(2T,x;T,z)\mu_{r}(dz)+\int_{\rn}\ln\varGamma(T,z;0,\xi)\mu_{r}(dz)\\
 & \geq\frac{n}{2}\ln r-C(\lambda)(\frac{T}{r}+r^{-n/q}T^{\frac{l-2}{l}}\Lambda^{2})-C(\frac{r}{T})^{n/2+1}\exp\left(\frac{\pi\vert\tilde{R}(T)\vert^{2}}{Cr}\right),
\end{align*}
i.e. 
\[
\varGamma(2T,x;0,\xi)\geq r^{\frac{n}{2}}\exp\left[-C(\frac{T}{r}+r^{-n/q}T^{\frac{l-2}{l}}\Lambda^{2})-C(\frac{r}{T})^{n/2+1}\exp\left(\frac{\pi\vert\tilde{R}(T)\vert^{2}}{Cr}\right)\right].
\]
Now we can take maximum of the right-hand side over all positive $r$.
Recall $\tilde{R}(t)=Ct^{(2-\gamma)/2}\ln\frac{1}{t}$, if we take
$r=\tilde{R}(T)^{2}$, then the right-hand side becomes
\[
T^{\frac{n}{2}(2-\gamma)}(\ln\frac{1}{T})^{\frac{n}{2}}\exp\left[-CT^{\theta_{1}}(\ln\frac{1}{T})^{-2}-CT^{\theta_{2}}(\ln\frac{1}{T})^{-2n/q}-CT^{\theta_{3}}(\ln\frac{1}{T})^{(n+2)}\right],
\]
where $\theta_{1}=\gamma-1$, $\theta_{2}=1-\frac{2}{l}-\frac{n}{q}(2-\gamma)$,
and $\theta_{3}=(\frac{n}{2}+1)(1-\gamma)<0$. Clearly $\theta_{3}=\min\{\theta_{1},\theta_{2},\theta_{3}\}<0$.
Because we are considering only for small $t$, the dominant term
will be 
\[
\varGamma(2T,x;0,\xi)\geq\exp\left[-CT^{\theta_{3}}(\ln\frac{1}{T})^{(n+2)}\right],
\]
and the proof is complete.
\end{proof}
\textcolor{black}{We can use the Chapman-Kolmogorov equation to obtain
a positive lower bound on the whole space, but we prefer to omit the
details of computations. }
\begin{rem}
\textcolor{black}{Using full power of the }Poincaré-Wirtinger\textcolor{black}{{}
inequality and following similar arguments as above, we can actually
drop the assumptions that $q\geq2$ and $l\geq2$. We only need to
assume that $1\leq\gamma<2$ to obtain a lower bound.}
\end{rem}

\section{Application of the Aronson type estimate}

Assume that $b\in L^{\infty}(0,T;L^{n}(\rn))$. In the previous section,
we have proved the Aronson estimate 
\[
\frac{1}{Ct^{n/2}}\exp\left[-C\frac{\vert x-\xi\vert^{2}}{t}\right]\leq\Gamma(t,x;0,\xi)\leq\frac{C}{t^{n/2}}\exp\left[-\frac{\vert x-\xi\vert^{2}}{Ct}\right]
\]
where $C$ only depends on $(n,\lambda,\Lambda)$. As an application
of the Aronson type estimate, we will prove the uniqueness of Hölder
continuous weak solutions.

We denote the parabolic ball as $Q((t_{0},x_{0}),R)=(t_{0}-R^{2},t_{0})\times B(x_{0},R)$
and 
\[
\underset{Q((t_{0},x_{0}),R)}{Osc}u=\max_{Q((t_{0},x_{0}),R)}u-\min_{Q((t_{0},x_{0}),R)}u.
\]
Firstly, we will need Nash's continuity theorem.
\begin{thm}
\label{thm: 4. nash's continuity theorem}Suppose $u\in C^{1,2}(Q((t_{0},x_{0}),R))$
is a solution, then for any $\delta\in(0,1)$, there are $\alpha\in(0,1]$
and $C>0$ depending only on $(\delta,n,\lambda,\Lambda)$ such that
\[
\vert u(t_{1},x_{1})-u(t_{2},x_{2})\vert\leq C\left(\frac{\vert t_{1}-t_{2}\vert^{1/2}\vee\vert x_{1}-x_{2}\vert}{R}\right)^{\alpha}\underset{Q((t_{0},x_{0}),R)}{Osc}u
\]
for any $(t_{1},x_{1}),(t_{2},x_{2})\in Q((t_{0},x_{0}),\delta R)$.
\end{thm}

The proof of this theorem is in the Appendix. Applying this theorem
to the fundamental solution, we have the following corollary.
\begin{cor}
\label{cor: holder of FS}There exist $\alpha\in(0,1]$ and $C>0$
depending only on $(n,\lambda,\Lambda)$ such that for any $\delta>0$,
we have 
\[
\vert\Gamma(t_{1},x_{1};0,\xi_{1})-\Gamma(t_{2},x_{2};0,\xi_{2})\vert\leq\frac{C}{\delta^{n}}\left(\frac{\vert t_{1}-t_{2}\vert^{1/2}\vee\vert x_{1}-x_{2}\vert\vee\vert\xi_{1}-\xi_{2}\vert}{\delta}\right)^{\alpha}
\]
for all $(t_{1},x_{1},\xi_{1}),(t_{2},x_{2},\xi_{2})\in[\delta^{2},\infty)\times\rn\times\rn$
with $\vert x_{1}-x_{2}\vert\vee\vert\xi_{1}-\xi_{2}\vert\leq\delta$.
\end{cor}

Now we are well prepared for proving the uniqueness and Hölder continuity
of the weak solutions. Given any $(a,b)$ satisfying conditions (\ref{eq: 1. ellipticity}),
(\ref{eq: 1. divergence free}) and $b$ belonging to $L^{\infty}(0,T;L^{n}(\rn))$,
by mollification we can find a sequence of smooth $(a_{m},b_{m})$
such that they satisfy conditions (\ref{eq: 1. ellipticity}), (\ref{eq: 1. divergence free})
and $\Vert b_{m}\Vert_{L_{t}^{\infty}L_{x}^{n}}\leq\Vert b\Vert_{L_{t}^{\infty}L_{x}^{n}}$.
Moreover, $b_{m}$ are compactly supported in space, $a_{m}\rightarrow a$
in $L_{loc}^{p}([0,T]\times\rn)$ for any $1\leq p<\infty$ and $b_{m}\rightarrow b$
in $L_{t}^{p}L_{x}^{n}$ for any $1\leq p<\infty$. Denote their corresponding
fundamental solution as $\Gamma^{m}$, then they have a uniform Aronson
estimate, and hence the family of the associated fundamental solutions
are equi-continuous in $[\delta^{2},\infty)\times\rn\times\rn$ according
to Corollary \ref{cor: holder of FS}. Thus, by the Arzela-Ascoli
theorem, there is a sub-sequence of $\{\Gamma^{m}\}$ converging locally
uniformly to some $\Gamma$. Moreover, $\Gamma$ still satisfies the
same Aronson estimate, Hölder continuity and Chapman-Kolmogorov equation.
\begin{thm}
Suppose equation (\ref{eq: 1. problem equation}) satisfies conditions
(\ref{eq: 1. ellipticity}), (\ref{eq: 1. divergence free}), and
consider $b\in L^{\infty}(0,T;L^{n}(\rn))$. Given initial value $f\in L^{2}(\rn)$,
there exists a unique Hölder continuous weak solution $u(t,x)\in L^{\infty}(0,T;L^{2}(\rn))\cap L^{2}(0,T;H^{1}(\rn))$
which satisfies $u(0,x)=f(x)$. Moreover, $\frac{\partial u}{\partial t}\in L^{2}(0,T;H^{-1}(\rn))$
and 
\[
u(t,x)=\int_{\rn}\Gamma(t,x;\tau,\xi)u(\tau,\xi)d\xi.
\]
\end{thm}

\begin{proof}
Given $f\in L^{2}(\rn)$, denote $u^{m}(t,x)=\Gamma_{t}^{m}f(x)$
the solution corresponding to $(a_{m},b_{m})$ and $u(t,x)=\Gamma_{t}f(x)$.
Since $\Gamma^{m}\rightarrow\Gamma$ point-wise, the dominated convergence
implies $u^{m}\rightarrow u$ point-wise as well. Notice that we have
energy inequality 
\[
\Vert u^{m}(t,\cdot)\Vert_{L_{x}^{2}}^{2}+\lambda\int_{0}^{t}\Vert\nabla u^{m}(s,\cdot)\Vert_{L_{x}^{2}}^{2}ds\leq\Vert f\Vert_{L_{x}^{2}}^{2},
\]
which implies that $u^{m}$ are weakly compact in $L^{\infty}(0,T;L^{2}(\rn))\cap L^{2}(0,T;H^{1}(\rn))$.
So its weak limit must be $u$ as defined above. Since $u^{m}$ satisfies
the following identity
\begin{multline*}
\int_{0}^{T}\int_{\mathbb{R}^{n}}u^{m}(t,x)\frac{\partial}{\partial t}\varphi(t,x)\;dxdt-\int_{0}^{T}\int_{\mathbb{R}^{n}}\langle\nabla\varphi(t,x)\cdot a_{m}(t,x),\nabla u^{m}(t,x)\rangle\;dxdt+\\
\int_{0}^{T}\int_{\mathbb{R}^{n}}\nabla\varphi(t,x)\cdot b_{m}(t,x)u^{m}(t,x)\;dxdt=0
\end{multline*}
for $\varphi\in C_{0}^{\infty}([0,T]\times\mathbb{R}^{n})$, by taking
$m\rightarrow\infty$ we obtain 
\begin{multline*}
\int_{0}^{T}\int_{\mathbb{R}^{n}}u(t,x)\frac{\partial}{\partial t}\varphi(t,x)\;dxdt-\int_{0}^{T}\int_{\mathbb{R}^{n}}\langle\nabla\varphi(t,x)\cdot a(t,x),\nabla u(t,x)\rangle\;dxdt+\\
\int_{0}^{T}\int_{\mathbb{R}^{n}}\nabla\varphi(t,x)\cdot b(t,x)u(t,x)\;dxdt=0,
\end{multline*}
which implies that $u$ is a weak solution. Also we have 
\begin{multline*}
\left|\int_{0}^{T}\int_{\mathbb{R}^{n}}\langle\nabla\varphi(t,x)\cdot a(t,x),\nabla u(t,x)\rangle\;dxdt-\int_{0}^{T}\int_{\mathbb{R}^{n}}\nabla\varphi(t,x)\cdot b(t,x)u(t,x)\;dxdt\right|\\
\leq(\frac{1}{\lambda}+\Lambda)\Vert\varphi\Vert_{L^{2}(0,T;H^{1})}\Vert u\Vert_{L^{2}(0,T;H^{1})},
\end{multline*}
and so 
\[
\left|\int_{0}^{T}\int_{\mathbb{R}^{n}}u(t,x)\frac{\partial}{\partial t}\varphi(t,x)\;dxdt\right|\leq(\frac{1}{\lambda}+\Lambda)\Vert\varphi\Vert_{L^{2}(0,T;H^{1})}\Vert u\Vert_{L^{2}(0,T;H^{1})}.
\]
Now we obtain $\frac{\partial u}{\partial t}\in L^{2}(0,T;H^{-1}(\rn))$.
This allows us to take $\varphi=u$ to have the energy estimate 

\[
\Vert u(t,\cdot)\Vert_{2}^{2}+\lambda\int_{0}^{t}\Vert\nabla u(s,\cdot)\Vert_{2}^{2}ds\leq\Vert f\Vert_{2}^{2},
\]
and therefore we obtain the uniqueness of the weak solution in $L^{\infty}(0,T;L^{2}(\rn))\cap L^{2}(0,T;H^{1}(\rn))$.
Suppose $\Gamma^{m}\rightarrow\Gamma'$, which will define another
$u'$ and it satisfies all the results above. So $u=u'$ implies $\Gamma=\Gamma'$,
and we also have the uniqueness of the fundamental solution.
\end{proof}

\section*{Appendix}

Here we prove the Nash's continuity theorem. The proof is inspired
by \cite{stroock1988diffusion}, which was originally written in probability
language and relies heavily on the strong Markov property of the diffusion
process. Here we rewrite it using a PDE approach instead. 

We still assume $(a,b)$ to be smooth and consider the Dirichlet problem
on $[0,T]\times B(x_{0},R)$ for any fixed $x_{0}$ and $R>0$.
\begin{equation}
\partial_{t}u-\divg(a\cdot\nabla u)+b\cdot\nabla u=0\qquad\mbox{ in }(0,T]\times B(x_{0},R)\label{eq: 4. dirichlet problem}
\end{equation}
with $u(0,x)=f(x)$ and $u(t,x)=0$ for $x\in\partial B(x_{0},R)$.
Clearly there is a unique regular fundamental solution $\Gamma^{x_{0},R}(t,x;\tau,\xi)$
with $x,\xi\in B(x_{0},R)$. So for any $f\in C_{0}^{\infty}(B(x_{0},R))$
satisfying $f\geq0$, 
\[
\Gamma_{t}^{x_{0},R}f(x)=\int_{B(x_{0},R)}\Gamma^{x_{0},R}(t,x;0,\xi)f(\xi)d\xi
\]
is the unique strong solution to Dirichlet problem. We will prove
the following lower bound for $\Gamma^{x_{0},R}(t,x;\tau,\xi)$, which
is also interesting by its own. 
\begin{thm}
\label{thm: 4. local lower bound}For any $\delta\in(0,1)$, there
exists a constant $C=C(\delta,n,\lambda,\Lambda)$ such that 
\[
\Gamma^{x_{0},R}(t,x;\tau,\xi)\geq\frac{1}{C(t-\tau)^{n/2}}\exp\left(-C\frac{\vert x-\xi\vert^{2}}{t-\tau}\right)
\]
for any $t-\tau\in(0,R^{2}]$ and $x,\xi\in B(x_{0},\delta R)$.
\end{thm}

\begin{proof}
Without loss of generality, we will take $\tau=0$. For any $t>0$,
given $f\in C_{0}^{\infty}(B(x_{0},R))$ satisfying $f\geq0$, consider
$w(s,x)=\Gamma_{s}f(x)-\Gamma_{s}^{x_{0},R}f(x)-M$ for $s\in[0,t]$
where 

\[
M=\sup_{s\in[0,t],z\in B(x_{0},R)^{c}}\Gamma_{s}f(z).
\]
Then we notice that $w$ solves (\ref{eq: 4. dirichlet problem})
in $(0,t]\times B(x_{0},R)$ with the initial-boundary condition that
$w(0,x)\leq0$ for $x\in B(x_{0},R)$ and $w(s,x)\leq0$ for $s\in(0,t]$,
$x\in\partial B(x_{0},R)$. So the maximum principle implies that
$w\leq0$ in $(0,t]\times B(x_{0},R)$, which means $\Gamma_{t}^{x_{0},R}f(x)\geq\Gamma_{t}f(x)-M$.
Since this is true for any $f\in C_{0}^{\infty}(B(x_{0},\delta R))^{+}$
with $\delta\in(0,1)$, we have 

\begin{align*}
\Gamma^{x_{0},R}(t,x;0,\xi) & \geq\Gamma(t,x;0,\xi)-\sup_{s\in[0,t],z\in B(x_{0},R)^{c},y\in B(x_{0},\delta R)}\Gamma(s,z;0,y)\\
 & \geq\frac{1}{Ct^{n/2}}\exp\left(-C\frac{\vert x-\xi\vert^{2}}{t}\right)-\sup_{s\in[0,t]}\frac{C}{s^{n/2}}\exp\left(-\frac{(1-\delta)^{2}R^{2}}{Cs}\right)
\end{align*}
for any $x,\xi\in B(x_{0},\delta R)$. Consider the second term, set
$\tilde{t}=t/R^{2}$ and $\tilde{s}=s/R^{2}$
\begin{align*}
 & \sup_{s\in[0,t]}\frac{C}{s^{n/2}}\exp\left(-\frac{(1-\delta)^{2}R^{2}}{Cs}\right)\\
 & \quad=\frac{1}{2Ct^{n/2}}\exp\left(-C\frac{\vert x-\xi\vert^{2}}{t}\right)\sup_{s\in[0,t]}\frac{2C^{2}t^{n/2}}{s^{n/2}}\exp\left(-\frac{(1-\delta)^{2}R^{2}}{Cs}+C\frac{\vert x-\xi\vert^{2}}{t}\right)\\
 & \quad=\frac{1}{2Ct^{n/2}}\exp\left(-C\frac{\vert x-\xi\vert^{2}}{t}\right)\sup_{\tilde{s}\in[0,\tilde{t}]}\frac{2C^{2}\tilde{t}^{n/2}}{\tilde{s}^{n/2}}\exp\left(-\frac{(1-\delta)^{2}}{C\tilde{s}}+C\frac{\vert x-\xi\vert^{2}}{\tilde{t}R^{2}}\right).
\end{align*}
If $\vert x-\xi\vert^{2}\leq\frac{(1-\delta)^{2}R^{2}}{2C^{2}}$ and
$t\le R^{2}$, it implies
\[
\sup_{\tilde{s}\in[0,\tilde{t}]}\frac{2C^{2}\tilde{t}^{n/2}}{\tilde{s}^{n/2}}\exp\left(-\frac{(1-\delta)^{2}}{C\tilde{s}}+C\frac{\vert x-\xi\vert^{2}}{\tilde{t}R^{2}}\right)\leq\sup_{\tilde{s}\in[0,\tilde{t}]}\frac{2C^{2}}{\tilde{s}^{n/2}}\exp\left(-\frac{(1-\delta)^{2}}{2C\tilde{s}}\right),
\]
where $\frac{2C^{2}}{\tilde{s}^{n/2}}\exp\left(-\frac{(1-\delta)^{2}}{2C\tilde{s}}\right)\rightarrow0$
as $\tilde{s}\rightarrow0$. So we can take $\tilde{t}$ small enough
so that $RHS\leq1$ and hence we have 
\[
\Gamma^{x_{0},R}(t,x;0,\xi)\geq\frac{1}{2Ct^{n/2}}\exp\left(-C\frac{\vert x-\xi\vert^{2}}{t}\right)
\]
where $\max\{t,\vert x-\xi\vert^{2}\}\leq\epsilon^{2}R^{2}$ for some
small $\epsilon$ depending on $(\delta,n,\lambda,\Lambda)$. 

Now we use the Chapman-Kolmogorov equation to extend this to any $x,\xi\in B(x_{0},\delta R)$
and $t\in(0,R^{2}]$. First consider $\vert x-\xi\vert\geq\epsilon R$
and any $t$, we set $\xi_{m}=\xi+\frac{m}{k+1}(x-\xi)$, $B_{m}=B(x_{0},\delta R)\cap B(\xi_{m},\frac{\vert x-\xi\vert}{k+1})$,
$t_{m}=\frac{mt}{k+1}$. Then for any $z_{m}\in B_{m}$, we have $\vert z_{m}-z_{m-1}\vert\leq\frac{3\vert x-\xi\vert}{k+1}$.
So, to obtain $\vert z_{m}-z_{m-1}\vert\leq\epsilon R$ and $\vert t_{m}-t_{m-1}\vert\leq\epsilon^{2}R^{2}$,
we just need to choose $k\geq\frac{3}{\epsilon^{2}}$. Now one has

\begin{align*}
\Gamma^{x_{0},R}(t,x;0,\xi) & \geq\int_{B_{1}}\cdots\int_{B_{k}}\prod_{m=0}^{k}\Gamma^{x_{0},R}(t_{m+1},z_{m+1};t_{m},z_{m})dz_{1}\cdots z_{k}\\
 & \geq C(\frac{\vert x-\xi\vert}{k+1})^{nk}\left(\frac{(k+1)^{n/2}}{2Ct^{n/2}}\exp\left(-C\frac{\vert x-\xi\vert^{2}}{(k+1)t}\right)\right)^{k+1}\\
 & \geq C\frac{\vert x-\xi\vert^{nk}}{t^{nk/2}}\frac{1}{t^{n/2}}\exp\left(-C\frac{\vert x-\xi\vert^{2}}{t}\right)\\
 & \geq\frac{1}{Ct^{n/2}}\exp\left(-C\frac{\vert x-\xi\vert^{2}}{t}\right).
\end{align*}
The only case left now is the case where $\vert x-\xi\vert\leq\epsilon R$
and $t\geq\epsilon^{2}R^{2}$. Set $t_{m}$ as before, then

\begin{align*}
\Gamma^{x_{0},R}(t,x;0,\xi) & \geq C(\frac{\epsilon R}{k+1})^{nk}\left(\frac{(k+1)^{n/2}}{2Ct^{n/2}}\exp\left(-C\frac{(k+1)\vert x-\xi\vert^{2}}{t}\right)\right)^{k+1}\\
 & \geq\frac{1}{Ct^{n/2}}\exp\left(-C\frac{\vert x-\xi\vert^{2}}{t}\right),
\end{align*}
and the proof is complete.
\end{proof}
Now we give the proof of Nash's continuity theorem. First consider
a non-negative solution on a parabolic ball $u\in C^{1,2}([t_{0}-R^{2},t_{0}]\times\overline{B(x_{0},R)})$,
clearly we have 
\[
u(t,x)\geq\int_{B(x_{0},R)}\Gamma^{x_{0},R}(t,x;t_{0}-R^{2},\xi)u(t_{0}-R^{2},\xi)d\xi
\]
by the maximum principle. Then by Theorem \ref{thm: 4. local lower bound}
\begin{equation}
u(t,x)\geq\frac{1}{C\vert B(x_{0},\delta_{2}R)\vert}\int_{B(x_{0},\delta_{2}R)}u(t_{0}-R^{2},\xi)d\xi\label{eq: 4. supermean value}
\end{equation}
for any $(t,x)\in[t_{0}-\delta_{1}^{2}R^{2},t_{0}]\times\overline{B(x_{0},\delta_{2}R)}$,
$\delta_{1},\delta_{2}\in(0,1)$ and $C$ depending only on $\delta_{1},\delta_{2},n,\lambda,\Lambda$.
This estimate is called the super-mean value property. 
\begin{lem}
Suppose $u\in C^{1,2}(Q((t_{0},x_{0}),R))$ is a solution, then for
any $\delta\in(0,1)$, there is a $\theta=\theta(\delta,n,\lambda,\Lambda)\in(0,1)$
such that 
\[
\underset{Q((t_{0},x_{0}),\delta R)}{Osc}u\leq\theta\underset{Q((t_{0},x_{0}),R)}{Osc}u.
\]
\end{lem}

\begin{proof}
Let
\[
M(r)=\max_{Q((t_{0},x_{0}),r)}u,\qquad m(r)=\min_{Q((t_{0},x_{0}),r)}u,
\]
and consider $M(R)-u$ and $u-m(R)$, which are non-negative solutions.
Inequality (\ref{eq: 4. supermean value}) implies that 
\[
M(R)-M(\delta R)\geq\frac{1}{C\vert B(x_{0},\delta R)\vert}\int_{B(x_{0},\delta R)}M(R)-u(t_{0}-R^{2},\xi)d\xi
\]
and 
\[
m(\delta R)-m(R)\geq\frac{1}{C\vert B(x_{0},\delta R)\vert}\int_{B(x_{0},\delta R)}u(t_{0}-R^{2},\xi)-m(R)d\xi.
\]
The sum of above two inequalities gives us 
\[
[M(R)-m(R)]-[M(\delta R)-m(\delta R)]\geq\frac{1}{C}[M(R)-m(R)],
\]
which completes the proof.
\end{proof}
\begin{proof}
[Proof of Theorem \ref{thm: 4. nash's continuity theorem}]Denote
$l=\vert t_{1}-t_{2}\vert^{1/2}\vee\vert x_{1}-x_{2}\vert$. If $\frac{l}{R}\geq1-\delta$,
then it is easy to find $C$ and the proof is done. If $\frac{l}{R}<1-\delta$,
we choose integer $K$ such that $(1-\delta)^{K+1}\leq\frac{l}{R}<(1-\delta)^{K}$.
Assume $t_{1}\leq t_{2}$. Then 
\begin{align*}
\vert u(t_{1},x_{1})-u(t_{2},x_{2})\vert & \leq\underset{Q((t_{2},x_{2}),(1-\delta)^{K}R)}{Osc}u\leq\theta^{K-1}\underset{Q((t_{2},x_{2}),(1-\delta)R)}{Osc}u\\
 & \leq\theta^{K-1}\underset{Q((t_{0},x_{0}),R)}{Osc}u=\theta^{-2}(\theta^{K+1})\underset{Q((t_{0},x_{0}),R)}{Osc}u.
\end{align*}
Now we can find $\alpha$ such that $\theta=((1-\delta)\wedge\theta)^{\alpha}$,
which implies $\theta^{K+1}\leq(1-\delta)^{(K+1)\alpha}\leq(\frac{l}{R})^{\alpha}$.
The proof is complete.
\end{proof}
\medskip

\textbf{Acknowledgements}

The authors would like to thank Professor Zhen Lei at Fudan University
for bringing the attention of the first named author (Z. Qian) to
papers \cite{liskevich2004extra} and \cite{zhang2004strong}.

\bibliographystyle{plain}
\nocite{*}
\bibliography{Lnbound}

\end{document}